\documentclass[10pt,a4paper]{article}
\usepackage[T1]{fontenc}
\usepackage{graphicx}
\usepackage{subcaption}
\usepackage{float} 
\usepackage{mathtools}
\usepackage{amssymb}
\usepackage{amsthm}
\usepackage{amsmath}
\usepackage{thmtools}
\usepackage{xcolor}
\usepackage{nameref}
\usepackage{babel}
\usepackage{cleveref}
\usepackage{natbib}
\usepackage{pdfpages}
\usepackage[left=2.5cm, right=2.5cm, top=2.5cm, bottom=2.5cm]{geometry}

\newtheorem{theorem}{Theorem}
\numberwithin{theorem}{section}
\newtheorem{lemma}[theorem]{Lemma}
\newtheorem{corollary}[theorem]{Corollary}

\newtheorem{example}[theorem]{Example}
\theoremstyle{remark}

\newtheorem{assumption}{Assumption}

\newcommand{\R}{\mathbb{R}}
\newcommand{\E}{\mathbb{E}}
\newcommand{\V}{\mathsf{V}}
\renewcommand{\P}{\mathbb{P}}
\DeclareMathOperator*{\Lip}{Lip} 
\allowdisplaybreaks
\title{A Wasserstein perspective of Vanilla GANs}
\author{Lea Kunkel and Mathias Trabs}
\date{Karlsruhe Insitute of Technology}
\begin{document}
	\maketitle
\begin{center}
	\begin{minipage}{0.8\textwidth}
		The empirical success of Generative Adversarial Networks
		(GANs) caused an increasing interest in theoretical research. The
		statistical literature is mainly focused on Wasserstein GANs and generalizations
		thereof, which especially allow for good dimension reduction properties.
		Statistical results for Vanilla GANs, the original optimization problem, 
		are still rather limited and require assumptions such as
		smooth activation functions and equal dimensions of the latent space and the ambient space. To bridge this gap, we draw a connection from Vanilla GANs to
		the Wasserstein distance. By doing so,  
		existing results for Wasserstein GANs
		can be extended to Vanilla GANs. In particular, we obtain an oracle inequality
		for Vanilla GANs in Wasserstein distance. The assumptions of this oracle inequality are designed to be satisfied by network architectures commonly used in practice, such as feedforward ReLU networks. 
	By providing a quantitative result for the approximation of a Lipschitz function by a feedforward ReLU network with bounded Hölder norm, we conclude a rate of convergence for Vanilla GANs as well as Wasserstein GANs as estimators of the unknown probability distribution.
	\end{minipage}
\end{center}

\section{Introduction}
\textit{Generative Adversarial Networks} (GANs) have attracted much attention since their introduction by \cite{goodfellow}, initially due to impressive results in the creation of photorealistic images. Meanwhile, the areas of application have expanded far beyond this, and GANs serve as a prototypical example of the rapidly evolving research area of generative models. 

The \textit{Vanilla GAN} as constructed by \cite{goodfellow} relies on the minimax game 
\begin{equation}
	\inf_{G\in\mathcal{G}}\sup_{D\in\mathcal{D}} \E\big[\log D(X)+\log\big(1-D(G(Z))\big)\big],\label{eq:GAN}
\end{equation} to learn an unknown distribution $\P^*$ of the random variable $X$. The generator $G$ chosen from a set $\mathcal{G},$ applied to the latent random variable $Z$ aims to mimic the distribution of $X$ as closely as possible. The discriminator $D,$ chosen from a set $\mathcal{D},$ has to distinguish between real and fake samples.   

The optimization problem \eqref{eq:GAN} is motivated by the Jensen-Shannon divergence. Generalizations of the underlying distance have led to various extensions of the original GAN, such as $f$-GANs (\cite{Nowozin2016}). More famously, \emph{Wasserstein GANs} (\cite{Arjovsky}), characterized by
\begin{equation}\label{eq:WassersteinGAN}
	\inf_{G\in\mathcal{G}}\sup_{W \in \Lip(1)} \E\big[W(X)-W(G(Z))\big],
\end{equation}
are obtained by replacing the Jensen-Shannon divergence by the Kantorovich dual of the Wasserstein distance. Here, $\Lip(1)$ denotes the set of all Lipschitz continuous functions with a Lipschitz constant bounded by one. This approach can be generalized  using \textit{Integral Probability Metrics} (IPMs, \cite{Mueller1997}).  For the application to GANs, see for example \cite{Liang2018}.

The analysis of Wasserstein GANs can exploit the existing theory on the Wasserstein distance. The latter has a long record of research, particularly in the context of optimal transport \citep{Villani2008} but also in machine learning, see \cite{torres2021} for an overview. In contrast, Vanilla GANs and the Jensen-Shannon divergence have been studied less extensively, and fundamental questions have not been settled. In particular, all statistical results for Vanilla GANs require the same dimension of the latent space and the target space which is in stark contrast to common practice. Another algorithmic drawback of Vanilla GANs highlighted by \cite{Arjovsky_2} is that an arbitrarily large discriminator class prevents the generator from learning. The Jensen-Shannon divergence between singular measures is by definition maximal. Therefore, we cannot expect proofs of convergence in a dimension reduction setting. In practice, however, Vanilla GANs have worked in various settings. Thus, using neural networks as a discriminator class must be advantageous compared to the set of all measurable functions. This empirical fact is supported by the numerical results by \cite{farnia2018} who impose a Lipschitz constraint on the discriminator class. 

 A wide range of methods exists for measuring distances between probability measures with varying properties. In view of the manifold hypothesis for high-dimensional data, it is crucial that the selected metric can discriminate between different singular measures. This is not the case for the Total Variation distance, the Jensen-Shannon distance, or even stronger metrics where singular distributions always attain the maximal value of the distance. Conversely, it is advantageous to select a strong metric, as this yields immediate bounds in weaker metrics. In view of this tension we will analyze Vanilla GANs with respect to the Wasserstein-$1$ distance which is often used in the statistical as well as the machine learning literature.  A comprehensive overview of the advantages of the Wasserstein-$1$ distance over other measures that metrize weak convergence, we direct the reader to \cite[p. 98 f.]{Villani2008}. For an overview of distances between probability measures, we refer the reader to  \citet[Figure 1]{Gibbs2002}.

\medskip

\paragraph{Our contribution.}
Our work aims to bridge the gap in theoretical analysis between Vanilla GANs and Wasserstein GANs while addressing the theoretical limitations of the former ones. By imposing a Lipschitz condition on the discriminator class in \eqref{eq:GAN}, we recover Wasserstein GAN-like behavior. As a main result, we can derive an oracle inequality for the Wasserstein distance between the true data generating distribution and its Vanilla GAN estimate. In particular, this allows us to transfer key features, such as dimension reduction, known from the statistical analysis of Wasserstein GANs. We show that the statistical error of the modified Vanilla GAN depends only on the dimension of the latent space, independent of the potentially much larger dimension of the feature space $\mathcal X$. Thus Vanilla GANs can avoid the curse of dimensionality. Such properties are well known from practice, but cannot be verified by the classical Jensen-Shannon analysis. On the other hand the derived rate of convergence for the Vanilla GAN is slower than for Wasserstein GANs which is in line with the empirical advantage of Wasserstein GANs.
 
We then consider the most relevant case where the classes $\mathcal{G}$ and $\mathcal{D}$ are parameterized by neural networks. Using our previous results, we derive an oracle inequality that depends on the network approximation errors for the best possible generator and the optimal Lipschitz discriminator. To bound the approximation error of the discriminator, we replace the Lipschitz constraint on the networks with a less restrictive Hölder constraint. Building on \cite{Guehring2020}, we prove a novel quantitative approximation theorem for Lipschitz functions using ReLU neural networks with bounded Hölder norm. As a result we obtain the rate of convergence $n^{-\alpha/2d^*},\alpha\in(0,1),$ with latent space dimension $d^*\ge2$ for sufficiently large classes of networks. Additionally, our approximation theorem allows for an explicit bound on the discriminator approximation error for Wasserstein-type GANs, which achieve the rate $n^{-\alpha/d^*},\alpha\in(0,1)$. 

We use a simple illustrative example to assess the practical implications of our theoretical results. This example allows us to quantify the rate depending on the number of observations, the dimension reduction property, and the stabilizing effect of a Lipschitz-constrained discriminator class.

\paragraph{Related work.}
The existence and uniqueness of the optimal generator for Vanilla GANs is shown by \cite{Biau2018} under the condition that the class $\mathcal{G}$ is convex and compact. They also study the asymptotic properties of Vanilla GANs.  
\cite{Belomnestny2021} have shown a non-asymptotic rate of convergence in the Jensen-Shannon divergence for Vanilla GANs with neural networks under the assumption that the density of $\P^*$ exists and that the generator functions are continuously differentiable.

In practice, however, the \textit{Rectifier Linear Unit} activation function (ReLU activation function) is commonly used \citep[p.13]{Aggarwal2018}. The resulting neural network generates continuous piecewise linear functions. Therefore, the convergence rate of \cite{Belomnestny2021} combined with \cite{Belomnestny2023} is not applicable to this class of functions. 

The statistical analysis of Wasserstein GANs is much better understood. \cite{Biau2021_Wasserstein} have studied optimization and asymptotic properties. \cite{Liang2018} has shown error decompositions with respect to the Kullback-Leibler divergence, the Hellinger distance and the Wasserstein distance. The case where the unknown distribution lies on a low-dimensional manifold is considered in \cite{Schreuder2020} as well as \cite{tang2022}. The latter also derive minmax rates in a more general setting using the Hölder metric.  Assuming that the density function of $\P^*$ exists, \cite{Liang2017} has shown a rate of convergence in Wasserstein distance with ReLU activation function with a factor growing exponentially in the depth of the network. Theoretical results including sampling the latent distribution in addition to dimension reduction have been derived by \cite{huang2021}, who have also shown a rate of convergence in a slightly more general setting (using the Hölder metric) using ReLU networks whose Lipschitz constant grows exponentially in the depth.
A rate of convergence using the total variation metric and leaky ReLU networks has been shown in \cite{Liang2018}. 

Convergence rates with respect to the Wasserstein distance have been studied by \cite{chen2020} and \cite{Chae2022}. Up to a logarithmic factor, optimal rates in the Hölder metric were obtained by \cite{stephanovitch2023} using smooth networks. In a similar setting, \cite{chakraborty2024} discussed several methods for dimension reduction. Recently, \cite{Suh2024} have reviewed the theoretical advances in Wasserstein GANs.

Ensuring Lipschitz continuity of the discriminator class is the essential property of Wasserstein GANs. Lipschitz-constrained neural networks and their empirical success are subject of ongoing research \newline(\cite{khromov2023}, in context of GANs see \cite{than2021}). Implementations of the Lipschitz constrained discriminator have evolved from weight clipping \citep{Arjovsky} to  penalizing the objective function \citep{gulrajani2017,wei2018,zhou2019,petzka2017,miyato2018,Asokan2023}, which heuristically leads to networks with bounded Lipschitz constants.
\cite{farnia2018} use an objective function that combines Wasserstein and Vanilla GANs. 

\paragraph{Outline.} In \Cref{prelim} we introduce the Vanilla GAN distance, which characterizes the optimization problem \eqref{eq:GAN}. In \Cref{compartibility}, we investigate the relation between the Vanilla GAN distance and the Wasserstein distance. We show that the distances are compatible to each other while not being equivalent. Using this relation, we derive an oracle inequality for the Vanilla GAN in \Cref{lipschitz_vanilla}, where $\mathcal{G}$ is a nonempty compact set and $\mathcal{D}$ is a set of Lipschitz functions. We show that Vanilla GANs can avoid the curse of dimensionality. In \Cref{nn_vanilla} we consider the situation where $\mathcal{G}$ and $\mathcal{D}$ consist of neural networks. Here we relax the Lipschitz condition to a $\alpha$-Hölder condition and prove a quantitative result for the approximation of a Lipschitz function by a feedforward ReLU network with bounded Hölder norm. We then prove a convergence rate for the Vanilla GAN with network generator and discriminator. 
 In \Cref{sec:wasserstein_gan} we obtain a convergence rate for Wasserstein-type GANs with network generator and discriminator using our approximation result. This enables us to compare Vanilla GANs directly to Wasserstein GANs. In \Cref{sec:simulation} we illustrate our theoretical results in a numerical example based on synthetic data.
All proofs are deferred to \Cref{sec:proofs}.

\section{The Vanilla GAN distance}
\label{prelim}
Let us first fix some notation. 
We equip $\mathbb{R}^d$ with the $\ell_p$-norm $|x|_p$, $1\leq p \leq \infty$, denote the number of nonzero entries of a $k \times l$ matrix $A,$ where $ k,l \in \mathbb{N},$ by $|A|_{\ell^0}\coloneqq |\{(i,j)\colon A_{ij} \neq 0 \}|,$ and define the ceiling of $x \in \mathbb{R}$ as $\lceil x\rceil:=\min \{k \in \mathbb{Z} \mid k \geq x\}.$ For ease of notation we abbreviate for $x \in (0, \infty)$
\begin{align}\label{eq:Power}
	[x]^{1;1/2} \coloneqq \max\{ x, \sqrt{x}\}.
\end{align}

For $\Omega\subset \mathbb{R}^d$ and $f\colon \Omega \rightarrow \mathbb{R}$ we define 
$\|f\|_{\infty, \Omega}:=\operatorname{ess \ sup} \{|f(x)|: x \in \Omega\}. $
We denote the set of bounded Lipschitz functions by
\begin{align*}
	\Lip(L, B,\Omega) \coloneqq \Big\{f\colon \Omega \rightarrow \mathbb{R} \,\Big|\, \|f\|_{\infty,\Omega}\leq B, \
\; \frac{|f(x)-f(y)|}{|x-y|_p} \leq L, x, y \in \Omega \Big\}.
\end{align*}
The set of unbounded Lipschitz functions is abbreviated by $\Lip(L, \Omega) \coloneqq \Lip(L,\infty ,\Omega).$
By Rademacher's theorem, a Lipschitz function is differentiable almost everywhere. 
For $\alpha \in (0,1]$ we define the $\alpha$-Hölder norm by
\begin{equation}\label{eq:Hoeldernorm}
\|f\|_{\mathcal{H}^{\alpha}(\Omega)}:=\max \Big\{\|f\|_{{\infty}}, \; \underset{x, y \in \Omega}{\operatorname{ess\; sup}}\; \frac{|f(x)-f(y)|}{|x-y|_p^{\alpha}}\Big\}
\end{equation}
and the $\alpha$-Hölder ball of functions with Hölder constant less or equal than $\Gamma>0 $ as 
\begin{equation}\label{eq:HoelderBall}
\mathcal H^\alpha(\Gamma,\Omega):=\big\{f\colon\Omega\to \mathbb R~\big|~\|f\|_{\mathcal{H}^{\alpha}(\Omega)} < \Gamma\big\}.
\end{equation}
In particular, $\Lip(L, B, \Omega)\subseteq \mathcal{H}^{\alpha}(\max(L, 2B), \Omega) $ for any $\alpha \in (0,1).$ We omit the domain $\Omega$ in our notation if $\Omega=(0,1)^d$. 

 \medskip
 
We observe i.i.d. samples $X_1,..., X_n \sim \P^*$ with values in $\mathcal X\coloneqq (0,1)^d$. On another space  $\mathcal{Z} \coloneqq (0,1)^{d^*}$, called the \textit{latent} space, we choose a latent distribution $\mathbb{U}$. Unless otherwise specified, $X \sim \P^*$ and  $Z\sim \mathbb{U}$. We further assume that $\mathbb{P}^*$ and $\mathbb{U}$ have finite first moments.
Throughout, the generator class $\mathcal{G}$ is a nonempty set of measurable functions from $\mathcal Z$ to $\mathcal X$. For $G \in \mathcal{G}$ the distribution of the random variable $G(Z)$ is denoted by $\mathbb{P}^{G(Z)}.$ 

Typically the discriminator class consists of functions mapping to $\mathbb{R}$ concatenated to a sigmoid function that maps into $(0,1)$ to account for the classification task. This is especially the case for standard classification networks.
The most common sigmoid function used for this purpose is the logistic function
 $
 	x\mapsto(1+e^{-x})^{-1},
 $
 which we fix throughout. Together with a shift by $\log4$, we can rewrite the Vanilla GAN optimization problem \eqref{eq:GAN} as 
\begin{equation}
	\inf_{G\in\mathcal{G}} \V_{\mathcal{W}} (\P^*, \P^{G(Z)}) 
 \label{eq:Vanilla_logistic}
\end{equation}
in terms of the \textit{Vanilla GAN distance} between probability measures $\mathbb P$ and $\mathbb Q$ on $\mathcal{X}$
 \begin{align}
 	\label{shifted_Vanilla}
 	\V_{\mathcal{W}}(\P, \mathbb{Q})	 \coloneqq  \sup_{W \in \mathcal{W}} \E_{\substack{X \sim \P \\ Y \sim \mathbb{Q}}}\Big[-\log\Big(\frac{1+e^{-W(X)}}{2}\Big)-\log\Big(\frac{1+e^{W(Y)}}{2}\Big)\Big],
 \end{align}
where $\mathcal{W}$ is a set of measurable functions $W\colon\mathcal X\to\R$. 
As long as $0 \in \mathcal{W},$ we have that $\V_{\mathcal{W}}\geq 0.$ 

To choose the generator $\hat{G}_n$ as the empirical risk minimizer, the unknown distribution $\P^*$ in \eqref{eq:Vanilla_logistic} must be replaced by the empirical distribution $\P_n$ based on the observations $X_1,...,X_n$. In practice, the expectation with respect to $Z\sim\mathbb U$ is replaced by an empirical mean too, which we omit for the sake of simplicity. Along \cite{huang2021}, the next and all subsequent results easily extend to the corresponding setting. 

The following error bound in terms of the Vanilla GAN distance provides an error decomposition for the empirical risk minimizer of the Vanilla GAN.
\begin{lemma}\label{Vanilla_error_decomp} Assume that $\mathcal{G}$ is chosen such that a minimum exists. Let $\mathcal{W}$ be symmetric, that is, $W \in \mathcal{W}$ implies $-W\in \mathcal{W}.$ 
	For \begin{equation}\label{G_hat_standard}
		\hat{G}_n \in \underset{G \in \mathcal{G}}{\arg\min}\; \V_{\mathcal{W}} (\P_n, \P^{G(Z)})
	\end{equation}
	we have that
	\begin{align}\label{eq:Vanilla_error_decomp}
			\V_{\mathcal{W}} (\P^*, \P^{\hat{G}_n(Z)}) \leq &	\min_{G \in \mathcal{G}}\V_{\mathcal{W}} (\P^*, \P^{G(Z)})+2\sup_{W \in \Lip(1) \circ \mathcal{W}}\frac{1}{n}\sum_{i=1}^n\big(W(X_i)-\E[W(X)]\big).
	\end{align}
\end{lemma} 
The first term in \eqref{eq:Vanilla_error_decomp} is the error due to the approximation capabilities of the class $\mathcal{G}.$ The second term refers to the stochastic error due to the amount of training data. As $\mathcal{W}$ is symmetric, the stochastic error is non-negative. Both error terms depend on the discriminator class $\mathcal W$. 
Large discriminator classes lead to finer discrimination between different probability distributions and thus to a larger approximation error term. Similarly, the stochastic error term will increase with the size of $\mathcal W$. The cost of small classes $\mathcal W$ is a less informative loss function on the left side of \eqref{eq:Vanilla_error_decomp}.

\medskip

If $\mathcal{W}$ is the set of all measurable functions, the analysis by \citet[Theorem 1]{goodfellow} shows that the Vanilla GAN distance is equivalent to the Jensen-Shannon distance. \cite{Arjovsky_2} have elaborated on the theoretical and practical disadvantages of this case. Similar to the Total Variation distance or the Hellinger distance, the Jensen-Shannon divergence is not compatible with high-dimensional settings because it cannot distinguish between different singular measures. Therefore, we need a weaker distance and thus restrict $\mathcal{W}$.  

The key insight of Wasserstein GANs \eqref{eq:WassersteinGAN} is that this particular drawback of the Jensen-Shannon distance can be solved by the Wasserstein distance. 
The latter is a metric on the space of probability distributions with finite first moment and meterizes weak convergence in this space \cite[Theorem 6.9]{Villani2008}. 
Let $\P$ and $\mathbb{Q}$ be two probability distributions on the same measurable space $(\Omega, \mathcal{A})$, the \textit{Wasserstein}-$1$ distance is defined as
\begin{align}
	\label{Wasserstein_distance}
	\mathsf{W}_1(\P, \mathbb{Q})\coloneqq \sup_{\substack{W \in \Lip(1) \\ W(0) = 0} }\E_{\substack{X \sim \P \\ Y \sim \mathbb{Q}}}	[W(X)-W(Y)]  = \sup_{W \in \Lip(1)} \E_{\substack{X \sim \P \\ Y \sim \mathbb{Q}}}	[W(X)-W(Y)] .
\end{align}
Bounds for weaker metrics, such as the Kolmogorov or Levy metric, can be easily derived from the bounds in the Wasserstein metric under weak conditions, see e.g. \cite{Gibbs2002}.
\medskip

Therefore, we choose $\mathcal{W} = \Lip(L)$ for some $L\ge 1$ in \Cref{Vanilla_error_decomp}. In this case the following result shows that the existence of an empirical risk minimizer is guaranteed as soon as $\mathcal{G}$ is compact. 
\begin{lemma}\label{existence_minimizer} Assume  $\mathcal{G}$ is compact with respect to the supremum norm.
	The map $T \colon \mathcal{G} \rightarrow \mathbb{R}_{\geq 0},$\linebreak$  T(G) \coloneqq V_{\Lip(L)}(\P_n, \P^{G(z)})$ is continuous and $ \underset{\mathcal{G}}{\arg\min} V_{\Lip(L)}(\P^*, \P^{G(z)})$ is nonempty.
\end{lemma} 
Hence, we throughout assume the following:
\begin{assumption}\label{compact_assumption}
     $\mathcal{G}$ is compact with respect to the supremum norm.
\end{assumption}
In the context of neural networks the compactness assumption is satisfied for all practically relevant implementations. Furthermore, it should be noted that the aforementioned assumption is only required for the use of the minimizing argument.   
\section{From Vanilla to Wasserstein and back}
\label{compartibility}
Our subsequent analysis builds on the following equivalence result between the Vanilla GAN distance and the  Wasserstein distance with an additional $L^2$-penalty term on the discriminator. 
\begin{theorem}\label{L2_penalty}
    For $L>2$ and $B>0$ we have for probability measures $\P$ and $\mathbb{Q} $ on $\mathcal{X}$
    \begin{align*}
          \sup_{\substack{W \in \Lip(1, B^{\prime})\\ W(\cdot) > -\log(2-2/L)   }} &\Big\{\E_{\substack{X \sim \mathbb{P}\\ Y \sim \mathbb{Q}}}[W(X)-W(Y)] - \frac{L(L-1)}{2} \E_{X \sim \mathbb{Q}}[W(X)^2]\Big\}\\
          &\leq\V_{\Lip(L, B)} (\P, \mathbb{Q})\leq \sup_{\substack{W \in \Lip(L, B) \\ W(\cdot) >-\log(2)}} \Big\{\E_{\substack{X \sim \mathbb{P}\\ Y \sim \mathbb{Q}}}[ W(X) - W(Y)] - \frac{e^{B}}{(2e^B-1)^2} \E_{X \sim \mathbb{Q}}[W(X)^2]\Big\},
    \end{align*}
    where $B^{\prime} = \log((1+e^B)/2).$
\end{theorem}
\Cref{L2_penalty} reveals that the Vanilla GAN distance is indeed compatible with the Wasserstein distance and will allow us to prove rates of convergence of the Vanilla GAN with respect to the Wasserstein distance. In doing so, we need to investigate the consequences of the penalty term. An upper bound without the penalty term and independent of $B$ can be shown as in the proof of \Cref{equiv_fast}. For the lower bound, a similiar improvement cannot be expected in general as indicated in \Cref{bsp_lin}. However, the restriction to $\operatorname{Lip}(1,B^{\prime})$ has far less severe consequences than the corresponding restriction in the upper bound.

We can deduce from \Cref{L2_penalty} that the Vanilla GAN distance is bounded from above and below by the Wasserstein distance or the squared Wasserstein distance, respectively. 

\begin{theorem}
	\label[theorem]{equiv_fast}
	 Let $L > 2$ and$B \in [1,\infty].$ For probability measures $\P$ and $\mathbb{Q} $ on $\mathcal{X}$ we have 
  \begin{align*}
		\min\Big(c_1 	\mathsf{W}_1(\P, \mathbb{Q}) ,c_2\mathsf{W}_1(\P, \mathbb{Q})^2\Big) & \leq  \V_{ \Lip(L,B)} (\P, \mathbb{Q})  \leq   L\mathsf{W}_1(\P, \mathbb{Q}),
	\end{align*}
	where $c_1 = \frac{1}{2} \frac{\log(2-2/L)}{d^{1/p}}$ and $c_2 =  \frac{1}{2 d^{2/p}} \frac{1}{L(L-1)}$, setting $1/p=0$ if $p=\infty$.
\end{theorem}

The assumption $L>2$ is not very restrictive. In practically relevant cases, such as neural network discriminators, the Lipschitz constant is typically quite large. A higher Lipschitz constraint on the discriminator will subsequently result in a less stringent constraint on the neural network. However, an arbitrarily large Lipschitz constant is also undesirable, as the upper bound grows linearly in $L$.

More importantly, we observe a gap between $\mathsf{W}_1(\P, \mathbb{Q})^2$ in the lower bound and $\mathsf{W}_1(\P, \mathbb{Q})$ in upper bound when $\mathsf{W}_1(\P, \mathbb{Q})<1$ which is a consequence of the penalty term in \Cref{L2_penalty}. The following example indicates that this loss is unavoidable, by restricting the discriminator class to a subset of $\Lip(L)$.

\begin{example}\label{bsp_lin}
	For $\varepsilon, \gamma > 0, \gamma + \varepsilon < 1 $ let	$\P= \frac{1}{2} (\delta_{\gamma} + \delta_{\gamma+ \varepsilon} )$ and $\mathbb Q= \frac{1}{2} (\delta_{0} + \delta_{\varepsilon} ).$  
	The Wasserstein distance is then given by
	\begin{equation*}
		\mathsf{W}_1 (\P, \mathbb Q) = \gamma.
	\end{equation*}
	We consider the Vanilla GAN  distance using $L$-Lipschitz affine linear functions as discriminator, \linebreak $\V_{a\cdot +b}(\P, \mathbb Q),$ with $a, b \in \mathbb{R}$ and $|a|\leq L.$ Note that the class of affine linear functions can be represented by one layer neural networks (for a definition see \Cref{nn_vanilla}). The optimal $b$ can be calculated explicitly, the optimal $a$ can be determined numerically. Using the optimal slope $a$ and $b$ we obtain for $\gamma < \varepsilon,$ $\varepsilon = \frac{1}{4}$ and $a > 16$
	\begin{align*}
		\frac{\mathsf{W}_1(\P,  \mathbb Q)^2}{2} \leq \V_{a\cdot +b} (\P,  \mathbb Q) \leq a \cdot \mathsf{W}_1(\P,  \mathbb Q)^2.
	\end{align*}
    If $\gamma\ge \varepsilon$, then the optimal $a$ is $a= L$ and 
	\begin{align*}
		\log(2)\cdot\mathsf{W}_1(\P,  \mathbb Q)  \leq \V_{a\cdot +b} (\P,  \mathbb Q)  \leq a \cdot\mathsf{W}_1(\P,  \mathbb Q).
	\end{align*} 
   See \Cref{proof_bsp_lin} for more details on these calculations.
\end{example}

Wasserstein GANs, where the generator is chosen as the empirical risk minimizer of the Wasserstein distance \eqref{Wasserstein_distance}, achieve optimal convergence rates (up to logarithmic factors) with respect to the Wasserstein distance as proved by \cite{stephanovitch2023}, see also \Cref{sec:wasserstein_gan}. In view of \Cref{equiv_fast} we cannot hope that Vanilla GANs achieve the same rate even if we use a Lipschitz discriminator class. This is in line with the better performance of Wasserstein GANs in practice. However, \Cref{equiv_fast} allows us to study the behavior of Vanilla GANs in settings where the dimension of the latent space is smaller than  the dimension of the sample space, a setting that is excluded in all previous works on convergence rates for Vanilla GANs.

\section{Oracle inequality for Vanilla GANs in Wasserstein distance}\label{lipschitz_vanilla}

Our aim is to bound the Wasserstein distance between the unknown distribution $\P^*$ and the generated distribution $\P^{\hat{G}_n(Z)}$ using the empirical risk minimizer $\hat{G}_n$ of the Vanilla GAN. The following oracle inequality shows that imposing a Lipschitz constraint on the discriminator class does circumvent the theoretical limitations of Vanilla GANs which is caused by the Jensen-Shannon distance. Recall notation \eqref{eq:Power}.
\begin{theorem}
	\label[theorem]{dreiecksugls}
	Let  $L > 2$ and $B \in[1,\infty].$ For the empirical risk minimizer $\hat G_n$ from \eqref{G_hat_standard} with $\mathcal W=\Lip(L, B)$ 
	we have 
	\begin{align}
		\label{fehleraufteilung}
			\mathsf{W}_1(\P^*, \P^{\hat{G}_n(Z)}) \leq c \big[\inf_{G \in \mathcal{G}} \mathsf{W}_1(\P^*, \P^{G(Z)})\big]^{1;1/2} + (1+c) [\mathsf{W}_1(\P_n, \P^*)]^{1;1/2},
	\end{align} for some constant $c> 0$ depending on $d,p$ and $L.$
\end{theorem}
Note that the discriminator class $\Lip(L, B)$ admits no finite dimensional parameterization and is therefore not feasible in practice. We will return to this issue in \Cref{nn_vanilla}.
The terms in \eqref{fehleraufteilung} can be interpreted analogously to the  interpretation of the bound in \Cref{Vanilla_error_decomp}, but here we have an oracle inequality with respect to the Wasserstein distance. The first term is the approximation error. It is large when $\mathcal{G}$ is not flexible enough to provide a good approximation of $\P^*$ by $\P^{G(Z)}$ for some $G\in\mathcal G$. The second term refers to the stochastic error. With a growing number of observations the empirical measure $\P_n$ converges to $\P^*$ in Wasserstein distance, see \cite{Dudley1969}, and thus the stochastic error converges to zero. Together with the bounds on $\mathsf W_1(\P_n,\P^*)$ by \cite{Schreuder2020_2} we conclude the following:

\begin{corollary}
	\label[corollary]{rate_schreuder} 	Let $L > 2, B \in[1,\infty]$.
 The empirical risk minimizer  $\hat{G}_n $ from \eqref{G_hat_standard} with $\mathcal W=\Lip(L, B)$ satisfies for some constant $c> 0$ depending on $d,p$ and $L$ that
	\begin{align*}
		\E[\mathsf{W}_1(\P^*, \P^{\hat{G}_n(Z)})] \leq& 
  \inf_{G^*\colon \mathcal{Z} \rightarrow \mathcal{X}} 
  \Big\{ c[\mathsf{W}_1(\P^*, \P^{G^*(Z)})]^{1;1/2} + c [\inf_{G \in \mathcal{G}}\|G-G^*\|_{\infty}]^{1;1/2}  \Big\}\\
 &\qquad \qquad\qquad + c \begin{cases}n^{-1 /2 d}, & d>2, \\ n^{-1/4}(\log n)^{1/2}, & d=2, \\ n^{-1 / 4}, & d=1,\end{cases}
\end{align*} 
where the infimum is taken over all Borel measurable functions $G^*\colon \mathcal{Z} \rightarrow \mathcal{X}$.
\end{corollary}
If there is some $G^*$ such that $\P^* = \P^{G^*(Z)},$ which is commonly assumed in the GAN literature, see e.g. \cite{stephanovitch2023}, \ the first term vanishes and the approximation error is bounded by $\inf_{G \in \mathcal{G}}[\|G-G^*\|^{1;1/2}_{\infty}]$. In the bound of the stochastic error we observe the curse of dimensionality: For large dimensions $d$ the rate of convergence $n^{-1/2d}$ deteriorates. 

To allow for a dimension reduction setting, we adopt the miss-specified setting from \citet[Theorem 1]{vardanyan2023}. In this scenario we can conclude statistical guarantees for Vanilla GANs that are comparable to the results obtained for Wasserstein GANs by \citet[Theorem 2]{Schreuder2020}. 
In view of \Cref{L2_penalty} we expect a slower rate of convergence compared to Wasserstein GANs.

\begin{theorem}\label{dim_red}  Let $L > 2, B \in[1,\infty]$ and  $M>0$. The empirical risk minimizer $\hat{G}_n$  from \eqref{G_hat_standard} with $\mathcal W=\Lip(L, B)$ satisfies 
\begin{align*}
	\E[\mathsf{W}_1(\P^*, \P^{\hat{G}_n(Z)})] &\leq 
 \inf_{\substack{G^{*} \in \Lip(M, \mathcal{Z})}}
  \Big\{ c[\mathsf{W}_1(\P^*, \P^{G^*(Z)})]^{1;1/2} + c [\inf_{G \in \mathcal{G}}\|G-G^*\|_{\infty}]^{1;1/2}  \Big\}\\
 &\qquad\qquad\qquad+c \begin{cases}n^{-1 /2 d^*}, & d^*>2, \\ n^{-1/4}(\log n)^{1/2}, & d^*=2, \\ n^{-1 / 4}, & d^*=1,\end{cases}
\end{align*}
 for some constant $c$ depending $d^*, d, p, L$ and $M$.
\end{theorem}

The Wasserstein distance $\mathsf{W}_1(\P^{G^*(Z)}, \P^*)$ now includes an error due to the dimension reduction while the stochastic error is determined by the potentially much smaller dimension $d^*<d$ of the latent space. Compared to \Cref{rate_schreuder}, the only price for this improvement is the additional Lipschitz restriction on $G^*$. We observe a trade-off in the choice of $d^*$, since large latent dimensions reduce the approximation error for $\mathbb P^*$, but increase the stochastic error term.  Additionally, there is a trade-off in $M$. A larger constant $M$ results in a smaller value of $W_1(\P^*, \P^{G^*(Z)})$, but increases the constant $c$. If the unknown distribution $\mathbb{P}^*$ is supported on a lower dimensional subspace and there exists a $G^*\in \operatorname{Lip}(M, \mathcal{Z})$ such that $\mathbb{P}^{G^*(Z)} = \mathbb{P}^*$, then the rate of convergence is solely determined by the dimension $d^*$ of $\mathcal{Z}$. This is true for the smallest possible $d^*$ for which a perfect $G^*$ exists, as well as any $d^{**}$ larger than $d^*$.

In many applications, the smallest possible $d^*$ is unknown. \Cref{dim_red} covers both over- and underestimation of the true dimension of the lower dimensional subspace. If the choice of $d^*$ is too small, then $\mathsf{W}_1(\P^*, \P^{G^*(Z)})$ might not converge to $0,$ but the stochastic error still converges with the smaller rate $d^*.$ If $d^*$ is selected to be larger than the dimension of the lower dimensional subspace, then there could be a $G^* \in \operatorname{Lip}(M, \mathcal{Z})$ such that $\mathsf{W}_1(\P^*, \P^{G^*(Z)})=0$, but the stochastic rates converges only with rate $d^*.$  
In the special case that a function $G^*\in \operatorname{Lip}(M, \mathcal{Z}) $ exists such that $\mathsf{W}_1(\P^*, \P^{G^*(Z)})=0$ the rate $n^{-1/2d^*}$ is slower to the rate $n^{-1/d^*}$ obtained for the Wasserstein GAN by \cite{Schreuder2020}. This is in line with \Cref{equiv_fast} and \Cref{bsp_lin}.

However, \Cref{dim_red} reveals why Vanilla GANs do perform well in high dimensions  in the setting of an unknown distribution on a lower dimensional manifold. This phenomenon could not be explained in previous work on Vanilla GANs. \cite{Belomnestny2021} and \cite{Biau2018} both obtain rates in the Jensen-Shannon distance.

\section{Vanilla GANs with network discriminator}
\label{nn_vanilla}
In practice, both $\mathcal{G}$ and $\mathcal{D}$ are sets of neural networks. Our conditions on the generator class $\mathcal G$ are compactness and good approximation properties of some $G^*$ which is chosen such that $\P^{G^*(Z)}$ mimics $\P^*$. Since neural networks have a finite number of weights, and the absolute value of those weights is typically bounded, the compactness assumption is usually satisfied and neural networks enjoy excellent approximation properties, cf. \cite{DeVore2020}. 

The situation is more challenging for the discriminator class. So far, $\mathcal{D}$ was chosen as the set of Lipschitz functions concatenated to the logistic function. The Lipschitz property is crucial for proof of \Cref{dreiecksugls} and thus for all subsequent results. 

Controlling the Lipschitz constant while preserving the approximation properties is an area of ongoing research and is far from trivial. Without further restrictions on the class of feedforward networks, the Lipschitz constant would be a term that depends exponentially on the size of the network, see \citet[Theorem 3.2]{Liang2017}. Bounding the Lipschitz constant of a neural network is a problem that arises naturally in the implementation of Wasserstein GANs. \cite{Arjovsky} use weight clipping to ensure Lipschitz continuity. Later, other approaches such as gradient penalization (see \cite{gulrajani2017}, which was further developed by \cite{wei2018}, \cite{zhou2019}), Lipschitz penalization \citep{petzka2017}, or spectral penalization \citep{miyato2018} were introduced and have achieved improved performance in practice.

\medskip

To extend the theory from the previous section to neural network discriminator classes, we first generalize \Cref{dreiecksugls} from $\mathcal W=\Lip(L,B)$ to subsets $\mathcal W\subseteq \Lip(L,B)$. As a result there is an additional approximation error term that accounts for the smaller discriminator class. 

\begin{theorem}
		\label[theorem]{dreiecksugls_diskriminierer} Let $L>2, B\in[1,\infty]$.  The empirical risk minimizer  $\hat{G}_n $ from \eqref{G_hat_standard} with $\mathcal{W} \subseteq \Lip(L,B)$ satisfies 
	\begin{align*}
		\mathsf{W}_1(\P^*, \P^{\hat{G}_n(Z)})& \leq c\big[\inf_{G \in \mathcal{G}} \mathsf{W}_1(\P^*, \P^{G(Z)})\big]^{1;1/2} +c \big[\inf_{W^{\prime}  \in \mathcal{W}} \sup_{W \in \Lip(L,B)}\|W-W^{\prime }\|_{\infty}\big]^{1;1/2}\\
        &\qquad\qquad\qquad +c \big[\mathsf{W}_1(\P_n, \P^*)\big]^{1;1/2},
	\end{align*}
 for some constant $c>0$ depending on $d,p$ and $L.$
\end{theorem}
 The approximation error of the discriminator depends on the supremum norm bound $B$ of the functions in $\mathcal{W}$. While the statement remains true for $B = \infty$, when approximating the set $\operatorname{Lip}(L, B)$, this bound will be essential.
To apply this result, we must ensure that the Lipschitz constant of a set of neural networks $\mathcal W$ is uniformly bounded by some constant $L$. 
Adding penalties to the objective function of the optimization problem does not guarantee a fixed bound on the Lipschitz constant. Approaches such as bounds on the spectral or row-sum norm of matrices in feedforward neural networks ensure a bound on the Lipschitz constant, but lead to a loss of expressiveness when considering ReLU networks, even in very simple cases such as the absolute value, see \cite{huster2019} and \cite{anil2019}. On the other hand, \cite{eckstein2020} has shown that one-layer $L$ Lipschitz networks are dense (with respect to the uniform norm) in the set of all $L$ Lipschitz functions on bounded domains. While this implies that the discriminant approximation error converges to zero for growing network architectures, the density statement does not lead to a rate of convergence that depends on the size of the network. 

\cite{anil2019}, motivated by \cite{chernodub2016}, have introduced an adapted activation function, Group Sort, which leads to significantly improved approximation properties of the resulting networks. They show that networks using the Group Sort activation function are dense in the set of Lipschitz functions, but there is no quantitative approximation result. A discussion of the use of Group Sort in the context of Wasserstein GANs can be found in \cite{Biau2021_Wasserstein}.

To overcome this problem, we would like to approximate not only the optimal discriminating Lipschitz function from the Wasserstein optimization problem in the uniform norm, but also its (weak) derivative. This would allow us to keep the Lipschitz norm of the approximating neural network bounded. For networks with regular activation functions \cite{Belomnestny2023} have studied the simultaneous approximation of smooth functions and their derivatives. \cite{Guehring2020} have focused on ReLU networks and have derived quantitative approximation bounds in higher order Hölder and Sobolev spaces. As an intrinsic insight from approximation theory, the regularity of the function being approximated must exceed the regularity order of the norm used to derive approximation bounds. Therefore, we cannot expect to obtain quantitative approximation results for ReLU networks in Lipschitz norm without assuming the continuous differentiability of the approximated function. 

Unfortunately, the maximizing function of the Wasserstein optimization problem is in general just Lipschitz continuous. Since we cannot increase the regularity of the target function, we instead relax the Lipschitz assumption of the discriminator in \Cref{hoelder_oracle} to $\alpha$-Hölder continuity for $\alpha\in(0,1)$. This generalization in the context of Wasserstein GANs has recently been discussed by \cite{stephanovitch2023}. Recall the definition of the Hölder ball from \eqref{eq:HoelderBall}.

\begin{theorem}
\label{hoelder_oracle}
Let $L>2, B \in [1, \infty),$ $\Gamma > \max(L, 2B)$and $M>0.$
The empirical risk minimizer  $\hat{G}_n $ from \eqref{G_hat_standard} with $\mathcal{W} \subseteq \mathcal{H}^{\alpha}(\Gamma)$
satisfies 
\begin{align*}
  \E[  \mathsf{W}_1(\P^*, \P^{\hat{G}_n(Z)})]\leq &  \inf_{G^* \in \Lip(M, \mathcal{Z})}\Big\{ c\big[\inf_{G\in\mathcal G} \|G^* - G\|_{\infty}^{\alpha}\big]^{1;1/2}
  + c \big[\mathsf{W}_1(\mathbb{P^*}, \P^{G^*(Z)})^{\alpha}\big]^{1;1/2}\Big\}\\
  &+c \big[ \inf_{W \in \mathcal{W}} \sup_{W^* \in \Lip(L,B)}\|W-W^{* }\|_{\infty}\big]^{1;1/2}
  \\& + c   \begin{cases}n^{-\alpha /2 d^*}, & 2\alpha<d^*,  \\ n^{-1/4}(\log n)^{1/2}, & 2 \alpha=d^*,\\ n^{-1 / 4}, & 2\alpha>d^*,\end{cases}
\end{align*}
for some constant $c$ depending on $d^*,d,p,L,M$ and $\Gamma$.

\end{theorem}
 The lower bound on the Hölder constant of the discriminator class $\mathcal{W}$ is not overly restrictive when employing neural networks for this function class. Since $c$ is increasing in $\Gamma$, it is advantageous to control the value of $\Gamma$.

It remains to show that there are ReLU networks that satisfy the assumptions of \Cref{hoelder_oracle}. To this end, we build on and extend the approximation results by \cite{Guehring2020}.

To fix the notation we give a general definition of feedforward neural networks.
 Let $d, K, N_1,...,N_K \in \mathbb{N}$. 
  A function $\Phi \colon \mathbb{R}^d \rightarrow \mathbb{R}$ is a neural network with $K$ layers and $N_1+\dots+N_K$ neurons if it results for an argument $x\in \mathbb{R}^d$ from the following scheme:
\begin{equation}
	\begin{aligned}
		& x_0:=x, \\
		& x_k:=\sigma\left(A_{k} x_{k-1}+b_k\right), \quad \text { for } k=1, \ldots K-1, \\
		\Phi(x)=\,& x_K:=A_K  x_{K-1}+b_K ,
	\end{aligned}
\end{equation}
 where for $k \in \{1,...,K\},$ $A_k\in \mathbb{R}^{N_k \times N_{k-1}}$  and $b_k \in \mathbb{R}^{N_k}.$ $\sigma\colon \mathbb{R} \rightarrow \mathbb{R}$ is an element-wise applied arbitrary activation function.
 The number of nonzero weights of all $A_k, b_k$ is given by $\sum_{j=1}^K\left(|A_j|_{\mathcal{\ell}^0}+|b_j|_{\ell^0}\right)$. We focus on the ReLU activation function $\sigma(x) = \max(0,x).$  

\begin{theorem}
	\label[theorem]{approx_res_hoelder}
	Let $ L,B>0$, and $0 <\alpha < 1$. Then there are constants $C^{\prime}, C^{\prime \prime}, C^{\prime \prime\prime} >0$ depending on $d, L, \alpha$ and $B$ with the following properties:	
	For any $\varepsilon \in(0,1 / 2)$ and any $f \in \Lip(L,B)$, there is a ReLU neural network  $\Phi_{\varepsilon}$ with no more than $\lceil C^{\prime} \log_2(\varepsilon^{-\frac{1}{1-\alpha}})\rceil $ layers, $\lceil C^{\prime \prime} \varepsilon^{-\frac{d}{1-\alpha}}\log^2_2(\varepsilon^{-\frac{1}{1-\alpha}})\rceil $ nonzero weights and $\lceil C^{\prime \prime \prime}   \varepsilon^{-\frac{d}{1-\alpha}}(\log_2^2(\varepsilon^{-\frac{1}{1-\alpha}}) \vee \log_2(\varepsilon^{-\frac{1}{1-\alpha}}))\rceil$  neurons such that 
	$$
	\big\|\Phi_{\varepsilon}-f\big\|_{\infty} \leq \varepsilon\qquad\text{and} \qquad \Phi_{\varepsilon}  \in \mathcal{H}^\alpha\big(\max(L,2B)+ \varepsilon\big).
	$$ 
	
\end{theorem}
Since there are many different reasons why controlling the Hölder constant of neural networks is interesting (with stability probably being the most prominent one), \Cref{approx_res_hoelder} is of interest on its own.
Combining \Cref{hoelder_oracle} and \Cref{approx_res_hoelder} with a standard approximation result for the generator approximation error, such as \citet[Theorem 1]{Yarotsky2017}, leads to a rate of convergence. The networks in $\mathcal{G}$ approximating the function $G^*\in \operatorname{Lip}(M, \mathcal{Z})$ are only required to be measurable without any additional smoothness assumption.

\begin{corollary}\label{alles_eingesetzt}

	For $0 < \alpha < 1,$ $\Gamma> 5,$ $M>0,$ $d^*> 2\alpha$ and $n > 2^{\frac{2d^*}{\alpha}}$ choose  $\mathcal{G}$ as the set of ReLU networks with at most $\lceil c \cdot \log(n)\rceil$ layers, $\lceil c\cdot  n\log(n)\rceil $ nonzero weights and $\lceil c \cdot   n\log(n)\rceil $ neurons and $\mathcal{W}'$ as the set of ReLU networks with at most  $ \lceil c \cdot \log(n)\rceil$ layers, $\lceil c \cdot n^{\frac{\alpha}{2(1-\alpha)}} \log^2(n)\rceil$ nonzero weights and $\lceil c \cdot n^{\frac{\alpha}{2(1-\alpha)}} \log^2(n)\rceil$  neurons, where $c$ is a constant depending on $d, d^*, \Gamma, M$ and $\alpha$. Then the empirical risk minimizer  $\hat{G}_n $ from \eqref{G_hat_standard} with $\mathcal{W} = \mathcal W'\cap \mathcal{H}^{\alpha}(\Gamma)$ satisfies
	\begin{align*}
		\E\left[\mathrm{W}_1\left(\P^*, \P^{\hat{G}_n(Z)}\right)\right] \leq c \cdot  n^{-\alpha / 2 d^*}
  + c 
  \big[\inf_{G^*\in \Lip(M, \mathcal{Z})}\mathsf{W}_1(\mathbb{P^*}, \P^{G^*(Z)})^{\alpha}\big]^{1;1/2} .
	\end{align*}
	\end{corollary}
From \Cref{approx_res_hoelder} we know that the set $\mathcal W'\cap \mathcal{H}^{\alpha}(\Gamma)$ of ReLU networks of finite width and depth is nonempty. In practice, this corresponds to a discriminator network with a controlled Hölder constant. On a bounded domain, any Lipschitz function is a Hölder function.
\Cref{alles_eingesetzt} shows that Vanilla GANs with a Hölder regular discriminator class are theoretically advantageous. The Hölder parameter $\alpha$ can be chosen arbitrarily close to one. On the one hand this reveals why a Lipschitz regularization as implemented for Wasserstein GANs also improves the Vanilla GAN. An empirical confirmation can be found in \cite{zhou2019} and \Cref{sec:simulation}. On the other hand the corollary then requires more neurons in the discriminator than in the generator class which coincides with common practice. 

 The width of the generator networks in \Cref{alles_eingesetzt} can be improved by replacing $G^* \in \operatorname{Lip}(M, \mathcal{Z})$ by with $G^* \in \operatorname{C}^{n-1}(\mathcal{Z}), n \in \mathbb{N},$ whose $(n-1)$-th derivative is Lipschitz continuous with Lipschitz constant $M.$ Once more, this results in a trade-off, as  $\big[\mathsf{W}_1(\mathbb{P^*}, \P^{G^*(Z)})^{\alpha}\big]^{1;1/2} $ increases when $G^*$ is selected from a smaller set of functions. 

\section{Wasserstein GAN}
\label{sec:wasserstein_gan}
The same analysis can be applied to Wasserstein-type GANs. The constrained on the Hölder constant can be weakened, as we do not need \Cref{equiv_fast}. Note that this does not impact the rate, but the constant. Define the Wasserstein-type distance with discriminator class $\mathcal{W}$ as 
\begin{equation*}
    \mathsf{W}_{\mathcal{W}}(\mathbb{P}, \mathbb{Q}) = \sup_{W \in \mathcal{W}} \mathbb{E}_{\substack{X \sim \mathbb{P} \\ Y \sim \mathbb{Q}}}[W(X) - W(Y)]. 
\end{equation*}
The following theorem shows that by using Hölder continuous ReLU networks as the discriminator class, Wasserstein-type GANs can avoid the curse of dimensionality. Furthermore, this avoids the difficulties arising from the Lipschitz assumption of the neural network, as pointed out by \cite{huang2021}. 

\begin{theorem}\label{thm:wasserstein_gan} For $0 < \alpha < 1,$ $\Gamma > 1, M > 0$ and $d > 2\alpha$ and $n > 2^{\frac{d}{\alpha}}$ choose $\mathcal{G}$ as the set of ReLU networks with at  most $\lceil c\cdot \log(n)\rceil $ layers, $\lceil c\cdot n\log(n)\rceil $ nonzero  weights and $\lceil c\cdot n\log(n)\rceil $ neurons and $\mathcal{W}'$ as the set of ReLU networks with at most  $\lceil c\cdot \log(n)\rceil $ layers, $\lceil c \cdot n^{\frac{\alpha}{(1-\alpha)}} \log^2(n)\rceil $ nonzero weights and $\lceil c \cdot n^{\frac{\alpha}{(1-\alpha)}} \log^2(n)\rceil $ neurons, where $c$ is a constant depending on $d, d^*, \Gamma, M$ and $\alpha$. 
    The empirical risk minimizer with $\mathcal{W} = \mathcal W'\cap \mathcal{H}^{\alpha}(\Gamma)$  \begin{equation*}
        \hat{G}_n \in \underset{G \in \mathcal{G}}{\operatorname{argmin}}\;  \mathsf{W}_{\mathcal{W}}(\mathbb{P}_n, \mathbb{P}^{G(Z)})
    \end{equation*}
satisfies
    \begin{align*}
         \mathbb{E}[\mathsf{W}_1(\mathbb{P}^{*}, \mathbb{P}^{\hat{G}_n(Z)})] \leq c \cdot n^{-\frac{\alpha}{d^*}} + \inf_{G^* \in \operatorname{Lip}(M, \mathcal{Z})}\mathsf{W}_1(\mathbb{P}^{*}, \mathbb{P}^{G^*(Z)}) .
    \end{align*}

\end{theorem}

Compared to \Cref{alles_eingesetzt} the rate improves to $n^{-\alpha/d^*}$ for any $\alpha<1.$  The number of observations necessary for the theorem to hold, the size of the discriminator network and the lower bound for $\Gamma$ decrease. Note that $\Gamma$ does not effect the rate, but the constants. In case there exists a $G^*$ such that $\mathsf{W}_1(\mathbb{P}^{*}, \mathbb{P}^{G^*(Z)})=0,$ this upper bound coincides with the lower bound in \citet[Theorem 1]{tang2022} up to an arbitrary small polynomial factor.

Our rate does not depend exponentially on the number of layers like the results of \cite{Liang2017}, \cite{huang2021} and we use non-smooth simple ReLU networks compared to smooth ReQU networks in \cite{stephanovitch2023} or group sort networks in \cite{Biau2021_Wasserstein}.

\section{Numerical illustration}\label{sec:simulation}
 The results in \Cref{nn_vanilla} and \Cref{sec:wasserstein_gan} were obtained under the assumption that the discriminator class consists of Lipschitz networks. In the context of image generation, these findings align with the results of \cite{zhou2019}, \cite{miyato2018}, \cite{kodali2017}, and \cite{fedus2017}. Furthermore, \cite{fedus2017} demonstrated in a two-dimensional experiment that a Vanilla GAN with a gradient penalty (and, consequently, a lower Lipschitz constant) can be effective in scenarios where the measures $\mathbb{P}^*$ and $\mathbb{P}^{\hat{G}(Z)}$ are singular. 
 
This section presents a transparent and accessible example that confirms our theoretical findings and especially demonstrates how imposing a Lipschitz constant on the discriminator stabilizes the Vanilla GAN. Additionally, it demonstrates the capacity of the Vanilla GAN to detect a lower dimensional manifold. In order to monitor rates of convergence, it is necessary to at least approximately evaluate $\mathsf{W}_1(\mathbb{P}^*, \mathbb{P}^{\hat{G}_n(Z)})$. Therefore, we study the numerical performance of the Vanilla GAN in a simulation setting, where the true data distribution is known by construction.

In this work, the Wasserstein distance is employed as the metric for measuring the rate of convergence. In practice, the Wasserstein distance is only computable in the one-dimensional case. To investigate multivariate distributions, we approximated the Wasserstein distance by averaging the Wasserstein distance on the marginals.

In order to model the distribution $\mathbb{P}^*$ of a lower dimensional manifold, we employed a one-dimensional uniform distribution on the graph of the function $x\mapsto\sin(4 \pi x)$ on the diagonal of the two-dimensional unit cube, resulting in a three-dimensional distribution. For the latent distribution, we used the one-dimensional normal distribution. Consequently, the dimensions of the lower dimensional manifold and the latent space are identical.

For the discriminator, we used a neural network with four layers of width $128$ concatenated to a sigmoid function. For the generator, we used a neural network with three layers of width $64.$ In order to preserve as much alignment as possible with the theoretical result, we used plain ReLU activations.
Each training consisted of $30000$ training iterations. We used the Adam optimizer \citep{kingma2014}  with parameters $\beta_1 = 0.9$ and $\beta_2 =0.999$, and a learning rate of $\gamma =  0.0005$. When the number of observations exceeded $512,$ we used minibatches of that size in each iteration. We updated generator and discriminator alternating. 

Three snapshots of the training of the Vanilla GAN for samples sizes $n=100$ and $n=1000$ are given in \Cref{fig:normal_to_sine_1000}, respectively. 
The difference between \Cref{fig:normal_to_sine_100_30000} and \Cref{fig:normal_to_sine_1000_30000} is solely due to the number of observations. The observations in \Cref{fig:normal_to_sine_100_0} to \Cref{fig:normal_to_sine_100_30000} cover the manifold to a lesser extent than the observations in \Cref{fig:normal_to_sine_1000_0} to \Cref{fig:normal_to_sine_1000_30000}. This corresponds to a larger stochastic error.

\begin{figure}[t]
    \centering
    \begin{subfigure}[b]{0.32\textwidth}
        \centering
        \includegraphics[trim= {3.5cm 1cm 2cm 2cm},clip,width=\textwidth]{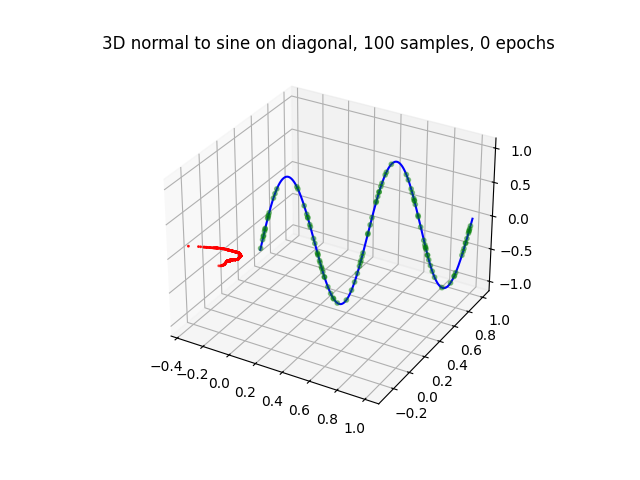}
        \caption{Step $0,$ $100$ obs.}
        \label{fig:normal_to_sine_100_0}
    \end{subfigure}
    \hfill
    \begin{subfigure}[b]{0.32\textwidth}
        \centering
        \includegraphics[trim={3.5cm 1cm 2cm 2cm},clip,width=\textwidth]{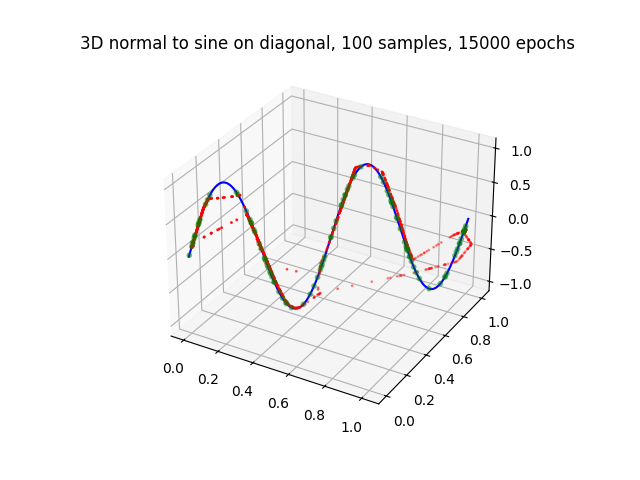}
        \caption{Step $15000,$ $100$ obs.}
        \label{fig:normal_to_sine_100_15000}
    \end{subfigure}
    \hfill
    \begin{subfigure}[b]{0.32\textwidth}
        \centering
        \includegraphics[trim={3.5cm 1cm 2cm 2cm},clip,width=\textwidth]{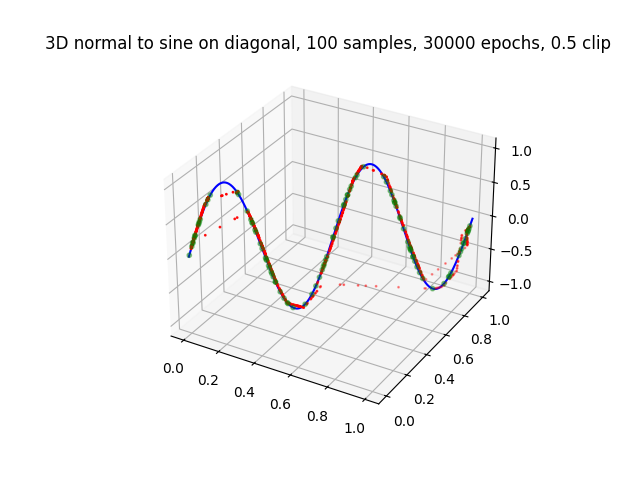}
        \caption{Step $30000,$ $100$ obs.}
        \label{fig:normal_to_sine_100_30000}
    \end{subfigure}
\vspace{1em}
    \begin{subfigure}[b]{0.32\textwidth}
        \centering
        \includegraphics[trim= {3.5cm 1cm 2cm 2cm},clip,width=\textwidth]{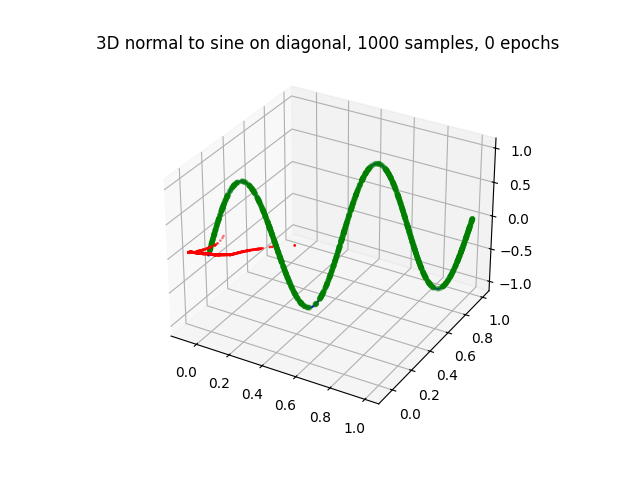}
        \caption{Step $0,$ $1000$ obs.}
        \label{fig:normal_to_sine_1000_0}
    \end{subfigure}
    \hfill
    \begin{subfigure}[b]{0.32\textwidth}
        \centering
        \includegraphics[trim={3.5cm 1cm 2cm 2cm},clip,width=\textwidth]{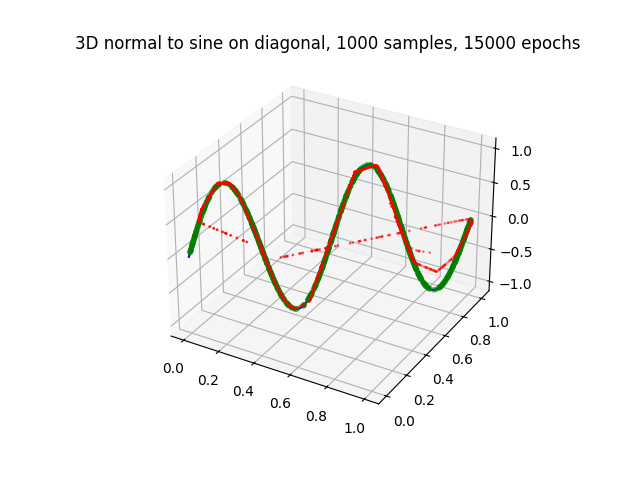}
        \caption{Step $15000,$ $1000$ obs.}
        \label{fig:normal_to_sine_1000_15000}
    \end{subfigure}
    \hfill
    \begin{subfigure}[b]{0.32\textwidth}
        \centering
        \includegraphics[trim={3.5cm 1cm 2cm 2cm},clip,width=\textwidth]{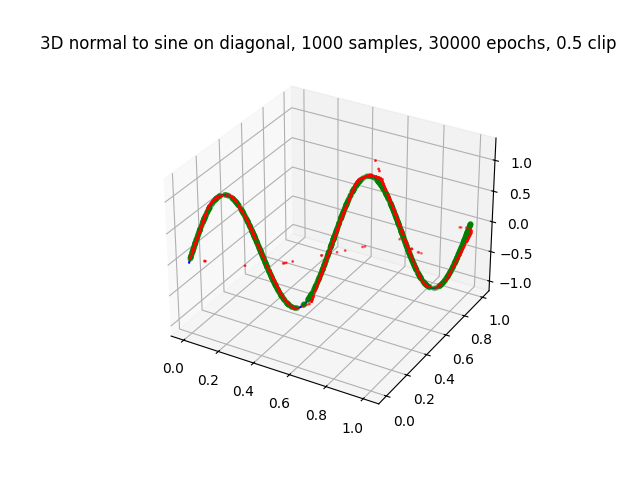}
        \caption{Step $30000,$ $1000$ obs.}
        \label{fig:normal_to_sine_1000_30000}
    \end{subfigure}
    \caption{Training of Vanilla GAN with weight clip using $100$ observations (first row) or $1000$ observations (second row). Red dots show $1000$ generated samples, green dots show the observations used for the training. The blue line is the one-dimensional manifold. }
  \label{fig:normal_to_sine_1000}
\end{figure}

To maintain the Lipschitz constant within a controllable range, we implemented  the  simple weight clipping mechanism of \cite{Arjovsky}, limiting each weight to a value of $0.5$. 
It is important to note that the network used in the unclipped case is also Lipschitz continuous, however, we do not have control over this Lipschitz constant. Given the width and depth parameters used in this study, it is evident that the Lipschitz constant of the clipped network remains relatively high and is considerably distinct from the theoretical value typically employed in Wasserstein GANs. However, a smaller Lipschitz constant requires an adjustment to the learning rate. Otherwise the weights are likely to remain at their maximum absolute value. This affects the experiment in several other ways. To ensure a fair and accurate comparison between the clipped and unclipped scenarios, we kept the learning rate consistent. 

The results are summarized in \Cref{fig:convergence_normal_to_sine}.
\begin{figure}[t]
    \centering
    \includegraphics[trim={0 0cm 0 1cm},clip, width=0.8\textwidth]{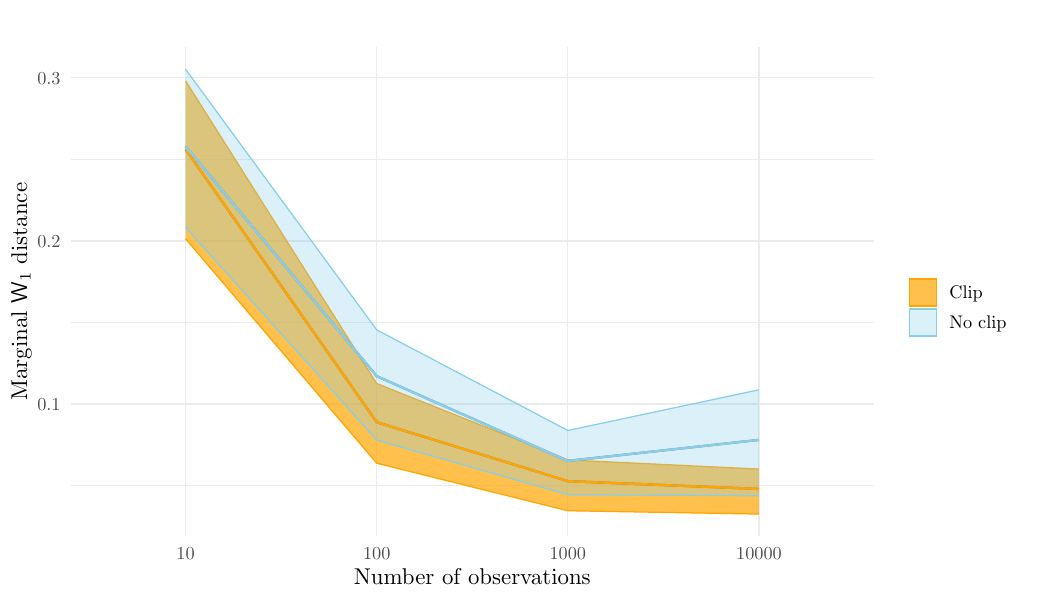}
    \caption{Marginal $\mathsf{W}_1$ distance depending on number of observations. Thick line shows the average over $50$ independent runs, ribbons show the first to third quartile.}
    \label{fig:convergence_normal_to_sine}
\end{figure}
As predicted by our theory, the averaged marginal Wasserstein distance between the generated distribution and the true data distribution decays approximately as $n^{-1/2}$ for $n \in \{10,100, 1000\}$. While we see a clear improvement with  $10000$ observations, the additional gain is limited by the optimization error, since the manifold is already densely covered for $1000$ observations.

It is apparent that controlling the Lipschitz constant overall stabilizes the training process, resulting in less variability in the results. In certain cases, the GAN without weight clipping can achieve the same level of effectiveness. This does not negate the outcome. Since the discriminator without clipped weights is still Lipschitz continuous (with a large Lipschitz constant), the theoretical limitations of Vanilla GANs without restricted discriminator classes do not directly translate to practice. This, combined with the finite nature of the implementations, ultimately resulted in the empirical success of these models. 
The variability between different simulation runs is described by the first to the third quartile in  \Cref{fig:convergence_normal_to_sine} which again confirms a more stable behavior of the clipped algorithm.

\begin{figure}[t]
    \centering
    \includegraphics[trim={0 0cm 0 1cm},clip, width=0.8\textwidth]{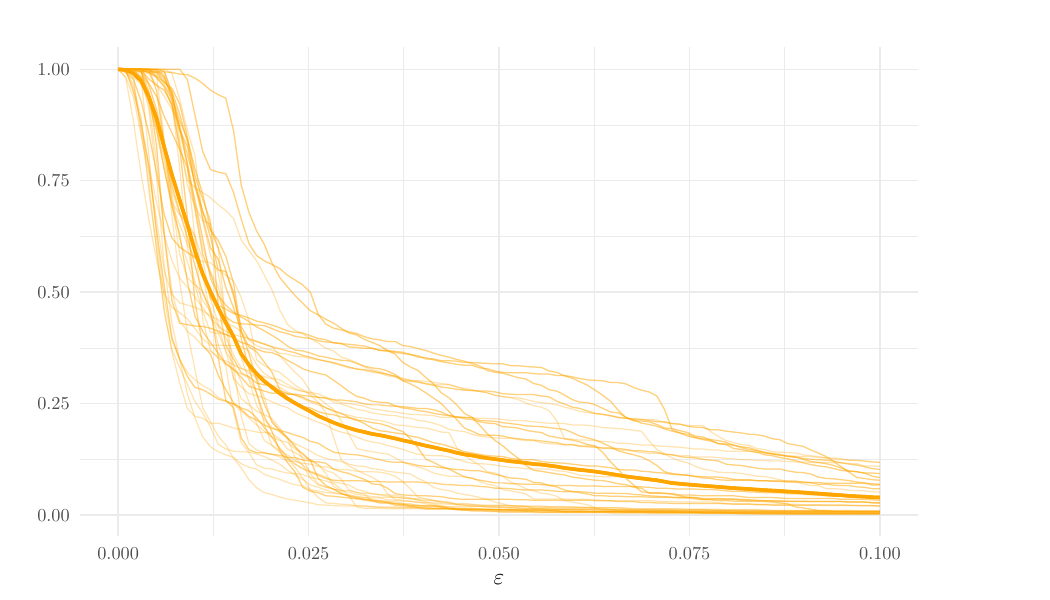}
    \caption{Percentage of generated samples with euclidean distance to manifold greater than $\varepsilon$ using $1000$ observations and a discriminator with $0.5$ clip. Transparent lines show the individual runs, thick line shows the average over $50$ runs.}
    \label{fig:deviations_normal_to_sine}
\end{figure}

\Cref{fig:deviations_normal_to_sine}  demonstrates the high degree of precision with which the generated samples concentrate on the low-dimensional support of the true data distribution. Our experiments show that this concentration holds true across all sample sizes and can be observed in both the clipped and unclipped case. However, a high concentration does not necessarily indicate that the generated distribution is an accurate imitation of the unknown distribution with respect to the Wasserstein distance. Consequently, \Cref{fig:deviations_normal_to_sine} is only informative in conjunction with \Cref{fig:convergence_normal_to_sine}.

Additionally, we investigated the use of a space $\mathbb{U}$ of the same dimension as the ambient space. Our observations indicated that the Vanilla GAN is still capable of identifying the lower dimensional subspace with reasonable efficacy.

\section{Discussion and limitations}
Our analysis demonstrates that GANs originally build on too sensitive distribution distances such as the Jensen-Shannon distance can be improved via a Lipschitz constraint in the discriminator class. This insight might be also applicable to other GANs, e.g. $f$-GANs by \cite{Nowozin2016}, which rely on a divgerence which cannot disciminate between different singular distributions and is thus not suitable for a dimension reduction setting. Overall we conclude that the choice of the discriminator class is much more important for the data generation capabilities than choice of the loss function which is typically dictated by some distance. Moreover, our analysis of the approximation error of the discriminator is not limited to the Vanilla GANs, but also applicable to optimal transport based GANs as demonstrated for the Wasserstein GAN.

There are several potential avenues for further development of the results presented in this paper.
This includes especially the mentioned potential implications to other types of GANs. While our analysis was limited to ReLU feedforward networks, one advancement of research on neural networks is the use of more sophisticated networks whose statistical analysis is not yet settled. In context of Wasserstein GANs, see for example \cite{Radford2015}.

Furthermore, inclusion of a bound on the Lipschitz constant (and not only the Hölder constant) would allows for direct application of \Cref{dreiecksugls_diskriminierer}, thereby eliminating the need to include the parameter $\alpha$ and thus improve the rates.
Additionally it would be interesting, whether there are conditions that allow for a faster rate of convergence for the Vanilla GAN in some cases (excluding scenarios as in \Cref{bsp_lin}).

The experiments also demonstrated that the GAN is capable of detecting data from a lower dimensional manifold if the latent space is of the same dimension as the ambient space. The proof of \Cref{dim_red} is contingent upon the dimension of the latent space. If the dimension of the latent space is chosen to be too small, then $\inf_{G^* \in \operatorname{Lip}(M, \mathcal{Z})} \mathsf{W}_1(\mathbb{P}^*, \mathbb{P}^{G^*(Z)})$ will be large. If the dimension of the latent space is chosen too large, this does not result in a deterioration of $\inf_{G^* \in \operatorname{Lip}(M, \mathcal{Z})} \mathsf{W}_1(\mathbb{P}^*, \mathbb{P}^{G^*(Z)})$ compared to the optimal choice. However, the rate that was proven then still depends on the higher latent dimension. Therefore rates that are adaptive to the unknown intrinsic dimension, potentially benefiting from results like \cite{Berenfeld2021}, would be interesting.

Bounds in other distances suitable to the dimension reduction setting are also of high interest. For example, the Wasserstein-$2$ metric is slightly stronger than the Wasserstein-$1$ metric (in the sense that $ \mathsf{W}_1 \leq \mathsf{W}_2$, see \cite[Remark 6.6]{Villani2008}). Our proofs rely on the duality of the Wasserstein-$1$ distance, hence they cannot be translated directly to the Wasserstein-$2$ distance. 

Finally, the objective of this study was to examine statistical perspectives, and thus, the optimization problem was not addressed. In the proofs, we employ the global minimizer and maximizer. Since we face a non-convex optimization problem gradient based methods may suffer from a considerable optimization error, especially for high-dimensional parameter spaces. Incorporating this optimization error would align more closely with real-world scenarios.

\appendix 
\section{Appendix}\label{sec:proofs}
\subsection{Proof for \Cref{prelim}}
\begin{proof}[Proof of \Cref{Vanilla_error_decomp}]
 Let $X\sim\P^*,\hat X\sim\P_n$ and $Z\sim\mathbb U$. The symmetry of $\mathcal{W}$ and the Lipschitz continuity of $x\mapsto\log(1+e^{-x})$ yields for	any $G \in \mathcal{G}$ 
	\begin{align*}
    &\V_{\mathcal W}(\P^*,\P^{\hat G_n}(Z))
	\\ &= \sup_{W \in \mathcal{W}}\E\Big[-\log\Big(\frac{1+e^{-W(X)}}2\Big)+\log\Big(\frac{1+e^{-W(\hat X)}}2\Big)-\log\Big(\frac{1+e^{-W(\hat X)}}2\Big)-\log\Big(\frac{1+e^{W(\hat{G}_n(Z))}}2\Big)\Big] \\
	& \leq \sup_{W \in \mathcal{W}}\E\big[-\log(1+e^{-W(X)})+\log(1+e^{-W(\hat X)})\big] + \V_{\mathcal W}(\P_n,\P^{\hat G_n(Z)}) \\
	& \leq \sup_{W \in \mathcal{W}}\E\big[-\log(1+e^{-W(X)})+\log(1+e^{-W(\hat X)})\big] + \V_{\mathcal W}(\P_n,\P^{G(Z)}) \\
	& = \sup_{W \in \mathcal{W}}\E[-\log(1+e^{-W(X)})+\log(1+e^{-W(\hat X)})] \\ & \quad + \sup_{W \in \mathcal{W}} \E[-\log(1+e^{-W(\hat X)}) + \log(1+e^{-W(X)})] +\V_{\mathcal W}(\P^*,\P^{G(Z)}) \\
&\leq 	2\sup_{W \in \Lip(1) \circ \mathcal{W}}\E[W(X)-W(\hat X)]+\V_{\mathcal W}(\P^*,\P^{G(Z)}) . 
	\end{align*} 
 The bound for $\hat G_n$ from \eqref{G_hat_standard} follows since $G \in \mathcal{G}$ was arbitrary.
\end{proof}
\begin{proof}[Proof of \Cref{existence_minimizer}]
	Let $(G_n)_{n \in \mathbb{N}}\in \mathcal{G}$ be a sequence that converges to $G\in \mathcal{G}.$  If $\V_{\Lip(L)}(\P^*, \P^{G(Z)}) \geq \V_{\Lip(L)}(\P^*, \P^{G_n(Z)})$, then
	\begin{align*}
	  	&\V_{\Lip(L)}(\P^*, \P^{G(Z)})-\V_{\Lip(L)}(\P^*, \P^{G_n(Z)})\\  
		&\leq 		\sup_{W \in \Lip(L)} \E\Big[\log\Big(\frac{1+e^{-W(G_n(Z))}}{2}\Big) - \log\Big(\frac{1+e^{-W(G(Z))}}{2}\Big) \Big]\\
		 & \leq \sup_{W \in \Lip(L)} \E\Big[W(G_n(Z)) - W(G(Z)) \Big]\\
		 & \leq L \| G_n - G \|_{\infty}.
	\end{align*}
	The case $\V_{\Lip(L)}(\P^*, \P^{G(Z)}) < \V_{\Lip(L)}(\P^*, \P^{G_n(Z)})$ can be bounded analogously. Therefore, $T$ is continuous and there is at least one minimizer if $\mathcal G$ is compact. 
\end{proof}

\subsection{Proofs for \Cref{compartibility}}
Before we prove the main results from \Cref{compartibility} we require an auxiliary lemma:
\begin{lemma}
	\label[lemma]{obere_schranke_allgemein}
	For $X \sim \P$ and $Y \sim \mathbb{Q}$ and an arbitrary set of measurable functions $\mathcal{W}$ we have that
	\begin{align*}
		\V_{\mathcal{W}}(\P,\mathbb Q)\leq 	\sup_{W \in \mathcal{W}} \E[-\log\big(1+e^{-W(X)}\big)+ \log\big(1+e^{-W(Y)}\big)].
	\end{align*}
\end{lemma}
\begin{proof}
	Since  \begin{align*}
		\log(1+e^{x}) + \log(1+e^{-x}) \geq \log(4) \quad  \text{ for all } x \in \mathbb{R},
	\end{align*} we can bound 
	\begin{align*}
		\sup_{W \in\mathcal{W}}& \E\Big[-\log\big(1+e^{-W(X)}\big)-\log\big(1+e^{W(Y)}\big)\Big]  \\
		& = 	\sup_{W \in \mathcal{W}} \E[-\log\big(1+e^{-W(X)}\big)+ \log\big(1+e^{-W(Y)}\big) - \log\big(1+e^{-W(Y)}\big)-\log\big(1+e^{W(Y)}\big)] \\
		& \leq 	\sup_{W \in \mathcal{W}} \E[-\log\big(1+e^{-W(X)}\big)+ \log\big(1+e^{-W(Y)}\big)]  \quad- \inf_{W \in \mathcal{W}}\E[ \log\big(1+e^{-W(Y)}\big)+\log\big(1+e^{W(Y)}\big)]  \\
		& \leq 	\sup_{W \in \mathcal{W}} \E[-\log\big(1+e^{-W(X)}\big)+ \log\big(1+e^{-W(Y)}\big)] - \log(4).\qedhere
	\end{align*}
\end{proof}
\begin{proof}[Proof of \Cref{L2_penalty}]
	  Defining 
   \begin{align*}
		\psi\colon \mathbb{R} \rightarrow \mathbb{R}, \ \psi(x) \coloneqq -\log\Big( \frac{1+e^{-x}}{2}\Big),
	\end{align*}
	we can rewrite
   \begin{align*}
		\V_{\Lip(L, B)} (\P, \mathbb{Q}) = 	\sup_{W \in \Lip(L, B)} \E[ \psi(W(X))+ \psi(-W(Y))].
	\end{align*}
 The function  $f\colon [-\log(2-2/L),\infty) \rightarrow \mathbb{R},$ $ \ f(x) =  \log(2e^x-1) $ is bijective and Lipschitz continuous with Lipschitz constant $L$ and satisfies $\psi(-f(x))=x$ for all $x\ge-\log(2-2/L)$. Therefore, we obtain a lower bound
 \begin{align*}
     \V_{\Lip(L, B)} (\P, \mathbb{Q}) &\geq \sup_{ \substack{W \in \Lip(1, \log((1+e^B)/2))\\ W(\cdot) \geq -\log(2-2/L)   }} \E[ \psi(f(W(X)))+ \psi(-f(W(Y)))]\\
      & = \sup_{\substack{W \in \Lip(1, B')\\  W(\cdot) \geq -\log(2-2/L)  }} \E[\psi(f(W(X)))-W(Y)].
 \end{align*}
 Since $f^{-1} \in  \Lip(1,\R)$, we can estimate $\V_{\Lip(L, B)}$ from above by
  \begin{align*}
     \V_{\Lip(L, B)} (\P, \mathbb{Q}) & =\sup_{W \in \Lip(L, B)} \E[ \psi(f(f^{-1}(W(X))))+ \psi(-f(f^{-1}(W(Y))))] \\
     & \leq \sup_{\substack{W \in \Lip(L, B) \\ W(\cdot ) >-\log(2) }}\E[ \psi(f(W(X)))+ \psi(-f(W(Y)))] \\
     & = \sup_{\substack{W \in \Lip(L, B) \\ W(\cdot ) >-\log(2)}} \E[ \psi(f(W(X))) - W(Y)].
 \end{align*}
 A Taylor approximation at zero of the function $\psi \circ f(x)=\log(2-e^{-x})$ yields that for every $x \in (-\log(2), \infty )$ there exists a $\xi $ between $x$ and $0$ such that
 \begin{align*}
     \psi \circ f (x) = x - \frac{e^{\xi}}{(2e^{\xi}-1)^2}x^2.
 \end{align*} 
 For the lower bound, we thus conclude
 \begin{align*}
     \V_{\Lip(L, B)} (\P, \mathbb{Q}) \geq  \sup_{\substack{W \in \Lip(1, B')\\  W(\cdot) \geq -\log(2-2/L)  }} \E[W(X)-W(Y)] - \frac{L(L-1)}{2} \E[W(X)^2].
 \end{align*}
 For the upper bound, we get
 \begin{equation*}
     \V_{\Lip(L, B)} (\P, \mathbb{Q}) \leq \sup_{\substack{W \in \Lip(L, B) \\ W(\cdot ) >-\log(2) }} \E[ W(X) - W(Y)] - \frac{e^{B}}{(2e^B-1)^2} \E[W(X)^2].\qedhere
 \end{equation*}
\end{proof}
Note that, using the function $g \colon (-\infty, \log(2-2/L)) \rightarrow \mathbb{R}, \; g(x) = -\log(2e^{-x}-1),$ we obtain lower and upper bounds with a penalty term depending on $\E[W(Y)^2]$ instead of $\E[W(X)^2]$.

\begin{proof}[Proof of \Cref{equiv_fast}]
	We prove the lower bound first. \Cref{L2_penalty} yields
\begin{align*}
    \V_{\Lip(L, B)} (\P, \mathbb{Q})& \geq    \sup_{\substack{W \in \Lip(1, B^{\prime})\\ W(\cdot) > -\log(2-2/L)   }} \E[W(X)-W(Y)]- \frac{L(L-1)}{2} \E[W(X)^2]\\
    & \geq  \sup_{ \substack{W \in \Lip( 1, \log(2-2/L)) }}\E[W(X)- W(Y)] - \frac{L(L-1)}{2} \E[W(X)^2] .
	\end{align*}
	Let $W^* \in \underset{ \substack{W \in \Lip( 1, \log(2-2/L)) }}{\arg\max} \E [W(X)-W(Y)].$  This element exists by \cite[Theorem 5.10 (iii)]{Villani2008}.
 Then $\delta W^*\in \Lip( 1, \log(2-2/L))$ for all $\delta \in (0,1]$ and we can conclude
	\begin{align*}
		\sup_{ \substack{W \in \Lip( 1, \log(2-2/L)) }}\E[W(X)&- W(Y)] - \frac{L(L-1)}{2} \E[W(X)^2] \\& \geq \sup_{\delta \in (0,1]} \Big\{\E[\delta W^*(X)- \delta W^*(Y)] - \frac{L(L-1)}{2} \E[(\delta W^*(X))^2]\Big\}\\
		& = \sup_{\delta \in (0,1]}\Big\{ \delta  \E[ W^*(X)-  W^*(Y)] - \delta^2 \frac{L(L-1)}{2}\E[( W^*(X))^2] \Big\},
	\end{align*}
    which is independent from $B$.
	In case $\Delta := \E[W^*(X)- W^*(Y)]  < L(L-1)\E[W^*(X)^2]$ we have for $\delta = \frac{\Delta }{\E[W^*(X)^2]L(L-1)}\in(0,1)$
	\begin{align*}
		&\sup_{ \substack{W \in \Lip( 1, \log(2-2/L)) }}\E[W(X)- W(Y)] - \frac{L(L-1)}{2} \E[W(X)^2]\\
        &\qquad\qquad\qquad\geq   \frac{\Delta^2 }{\E[W^*(X)^2]L(L-1)} -\frac{\Delta^2}{2\E[W^*(X)^2]L(L-1)}\\
		  &\qquad\qquad\qquad= \frac{\Delta^2}{2\E[W^*(X)^2]L(L-1)} \\
        &\qquad\qquad\qquad\ge\frac{\Delta^2}{2\log(2-2/L)^2L(L-1)},
	\end{align*}
    where we used $|W^*(x)|\le\log(2-2/L)$ in the last step. In case $\Delta \geq L(L-1)\E[W^*(X)^2]$ we obtain
	\begin{align*}
		\E[W^*(X)- W^*(Y)] - \frac{L(L-1)}{2}\E[W^*(X)^2]&\geq \frac{1}{2} \E[W^*(X)- W^*(Y)]. 
	\end{align*}
Using  the boundedness of $[0,1]^d$, we get \begin{align*}
    \Delta&=\sup_{ \substack{W \in \Lip( 1, \log(2-2/L)) }}\E [W(X)-W(Y)]\\
    &\geq  \sup_{ \substack{W \in \Lip(\log(2-2/L)d^{-1/p},\infty)  }}\E [W(X)-W(Y)]\\& = \frac{\log(2-2/L))}{d^{1/p}}\mathsf{W}_1(\P, \mathbb{Q}).
\end{align*}
  Hence we can conclude the claimed lower bound for 
 \begin{align*}
		c_1 =  \frac{1}{2} \frac{\log(2-2/L)}{d^{\frac{1}{p}}}, \quad 
		c_2 =  \frac{1}{2d^{\frac{2}{p}}L(L-1)}.
	\end{align*}
	 
	For the upper bound we use \Cref{obere_schranke_allgemein} with $\mathcal{W} =\Lip(L).$ Since for $W\in \Lip(L)$ the function $-\log\big(1+e^{-W(\cdot)} \big)\in\Lip(L) $ we conclude
	\begin{align*}
		\V_{\Lip(L,B)}(\P,\mathbb Q)&\le\sup_{W \in \Lip(L)} \E[\psi(W(X))+ \psi(W(Y))] \\
        & \leq 	\sup_{W \in \Lip(L)} \E[-\log\big(1+e^{-W(X)}\big)+ \log\big(1+e^{-W(Y)}\big)]\\& \leq
		\sup_{W \in \Lip(L)} \E[ W(X)-W(Y)]\\
		& = L\sup_{W \in \Lip(1)} \E[ W(X)-W(Y)].\qedhere
	\end{align*}
\end{proof}

\subsection{Proofs for \Cref{lipschitz_vanilla}}
\begin{proof}[Proof of \Cref{dreiecksugls}]
Using \Cref{equiv_fast} and the triangle inequality for the Wasserstein distance, we deduce for every $G \in \mathcal{G}$ and $c=\max(c_1^{-1},c_2^{-1/2})$ that
\begin{align*}
    \mathsf{W}_1(\P^*, \P^{\hat{G}_n(Z)})& \leq \mathsf{W}_1(\P^*, \P_n)+ \mathsf{W}_1(\P_n, \P^{\hat{G}_n(Z)})\\
     & \leq \mathsf{W}_1(\P^*, \P_n)+ c \big[\V_{\Lip(L,B)}(\P_n, \P^{\hat{G}_n(Z)}) \big]^{1;1/2}\\
     & \leq \mathsf{W}_1(\P^*, \P_n)+ c \big[\V_{\Lip(L,B)}(\P_n, \P^{G(Z)}) \big]^{1;1/2}\\
     & \leq \mathsf{W}_1(\P^*, \P_n)+ cL \big[\mathsf{W}_{1}(\P_n, \P^{G(Z)}) \big]^{1;1/2}\\
     & \leq  (1+cL) \big[\mathsf{W}_1(\P^*, \P_n)\big]^{1;1/2}+ cL \big[\mathsf{W}_{1}(\P^* ,\P^{G(Z)}) \big]^{1;1/2}.
\end{align*}
As $G \in \mathcal{G}$ was arbitrary, we can choose the infimum over $\mathcal{G}.$
\end{proof}

\begin{proof}[Proof of \Cref{rate_schreuder}]
	For every measurable $G^* \colon \mathcal{Z} \rightarrow \mathcal{X}$ and any $G\in\mathcal G$ we have
  \begin{align*}
      \mathsf{W}_1(\P^*, \P^{\hat G_n(Z)}) 
      & \le \mathsf{W}_1(\P^*, \P^{G^*(Z)}) + \mathsf{W}_1(\P^{G^*(Z)}, \P^{G(Z)})\\
      &= \mathsf{W}_1(\P^*, \P^{G^*(Z)}) +\sup_{W \in \Lip(1)} \E[W(G^*(Z)) - W(G(Z))]\\
      &\le \mathsf{W}_1(\P^*, \P^{G^*(Z)}) +\E[|G^*(Z) - G(Z)|_p]\\
      &\le \mathsf{W}_1(\P^*, \P^{G^*(Z)}) +\| G^* - G\|_{\infty}.
  \end{align*}
    Since $G^*$ was arbitrary \Cref{dreiecksugls} yields for some constant $c$
	\begin{align*}
		\E[\mathsf{W}_1(\P^*, \P^{G(Z)})] \leq & c \cdot \E[ \max(\sqrt{\mathsf{W}_1(\P_n, \P^*)} , \mathsf{W}_1(\P_n, \P^*) )] \\
  &+ c \cdot  \inf_{\substack{G^{*}\colon \mathcal{Z} \rightarrow \mathcal{X}}}
  \Big\{ [\mathsf{W}_1(\P^*, \P^{G^*(Z)})]^{1;1/2} + [\inf_{G\in\mathcal G}\|G-G^*\|^{1;1/2}_{\infty}]  \Big\}
	\end{align*}	
  Here the infimum can be used as we can increase the constant $c$ multiplied to both terms by an arbitrary small $\varepsilon>0$ to account for the possibly infinitesimal smaller value.
	Using Jensen's inequality, we can bound the stochastic error term by
    \begin{align*}		
		\E[ \max(\sqrt{\mathsf{W}_1(\P_n, \P^*)} , \mathsf{W}_1(\P_n, \P^*) )]& \leq   \E[\sqrt{\mathsf{W}_1(\P_n, \P^*)}] +  \E[ \mathsf{W}_1(\P_n, \P^*) ] \\
		& \leq   \sqrt{\E[\mathsf{W}_1(\P_n, \P^*)]} +  \E[ \mathsf{W}_1(\P_n, \P^*) ] .
	\end{align*}	
	From \citet[Theorem 4]{Schreuder2020_2} we know
	\begin{align*}
		\E[\mathsf{W}_1(\P^*, \P_n)] \leq c^{\prime} \begin{cases}n^{-1 / d}, & d>2 \\ n^{-1 / 2} \log (n), & d=2 \\ n^{-1 / 2}, & d=1.\end{cases}
	\end{align*} where $c$ depends only on $d$.
	Since $(\log n)/\sqrt n\le1$, we conclude
	\begin{align*}
		\sqrt{\E[\mathsf{W}_1(\P_n, \P^*)]} +  \E[ \mathsf{W}_1(\P_n, \P^*) ] \leq 2 c^{\prime} \begin{cases}n^{-1 /2 d}, & d>2 \\ n^{-1/4}(\log n)^{1/2}, & d=2 \\ n^{-1 / 4}, & d=1.\end{cases}
	\end{align*}
	
\end{proof}

\begin{proof}[Proof of \Cref{dim_red}]

	With the same reasoning as in the proof of \Cref{rate_schreuder},  there exists some $c$ such that for any measurable $G^*\colon\mathcal Z\to\mathcal X$ and any $G\in\mathcal G$
	\begin{align*}
		\E[\mathsf{W}_1(\P^*, \P^{G(Z)})] & \leq  c \big( \sqrt{\E[\mathsf{W}_1(\P_n, \P^*)]} +  \E[ \mathsf{W}_1(\P_n, \P^*) ]  + [\mathsf{W}_1(\P^*, \P^{G^*(Z)})]^{1;1/2}+   [\inf_{G \in \mathcal{G}} \| G^* - G\|_{\infty} ]^{1,1/2}\big)
	\end{align*}
By the triangle inequality \begin{align*}
    \mathsf{W}_1(\P_n, \P^*) \leq \mathsf{W}_1(\P_n, \P^{G^*(Z)}) + \mathsf{W}_1( \P^{G^*(Z)}, \P^*) .
\end{align*}
Let $Z_i \sim \mathbb{U}$ be i.i.d. random variables and denote the corresponding empirical measure by $\mathbb U_n$. For $G^*\in\Lip(M,\mathcal Z)$ we can then bound the first term by
\begin{align}
    \mathsf{W}_1(\P_n, \P^{G^*(Z)}) & = \sup_{W \in \Lip(1)}\frac{1}{n} \sum_{i = 1}^{n} W(X_i) - \E[W \circ G^{*}(Z)] \notag\\
     &  \leq  \sup_{W \in \Lip(1)}\frac{1}{n} \sum_{i = 1}^{n} |W(X_i)   - W\circ G^{*}(Z_i)| + \sup_{W \in \Lip(1)} \frac{1}{n}\sum_{i = 1}^{n}  W\circ G^{*}(Z_i)- \E[W \circ G^{*}(Z)] \notag\\
     & \leq  \frac{1}{n} \sum_{i = 1}^{n}| X_i   -  G^{*}(Z_i) |_p + \sup_{f \in \Lip(M)} \frac{1}{n}\sum_{i = 1}^{n} f(Z_i)
     - \E[f(Z)] \label{eq:proofDimRed}\\
      & = \frac{1}{n} \sum_{i = 1}^{n}| X_i   -  G^{*}(Z_i) |_p + M \cdot \mathsf W_{1}(\mathbb{U}_n, \mathbb{U})\notag
\end{align}
Hence,
\begin{align*}
    \E\big[\mathsf{W}_1(\P_n, \P^{G^*(Z)}) \big] \leq \frac{1}{n}\sum_{i = 1}^{n} \E[| X_i   -  G^{*}(Z_i) |_p] + M\cdot  \E[ W_{1}(\mathbb{U}_n, \mathbb{U})].
\end{align*}
Note that $\E[| X_i   -  G^{*}(Z_i) |_p]= \mathsf{W}_1(\P^{G^*(Z)}, \P^*) $ by the duality formula of $\mathsf{W}_1$ used in this work, see \citet[ Definition 6.2 and Remark 6.5]{Villani2008}. For $\E[ W_{1}(\mathbb{U}_n, \mathbb{U})]$, we can exploit the convergence rate for the empirical distribution as in \Cref{rate_schreuder}, but now in the $d^*$-dimensional latent space $\mathcal Z$. Therefore, there exists a $c^{\prime}$ such that
	\begin{align*}
		\sqrt{\E[\mathsf{W}_1(\P_n, \P^*)]} +  \E[ \mathsf{W}_1(\P_n, \P^*) ] \leq  c^{\prime} [\mathsf{W}_1(\P^{G^*(Z)}, \P^*)]^{1;1/2} + c^{\prime} \begin{cases}n^{-1 /2 d^*}, & d^*>2 \\ n^{-1/4}(\log n)^{1/2}, & d^*=2 \\ n^{-1 / 4}, & d^*=1.\end{cases}
	\end{align*}
\end{proof}

\subsection{Proofs of \Cref{dreiecksugls_diskriminierer} and \Cref{hoelder_oracle}}
\begin{proof}[Proof of \Cref{dreiecksugls_diskriminierer}]
    First, we verify that for any two nonempty sets $\mathcal{W}_1$ and $\mathcal{W}_2$ we have
	\begin{align}
		\label{Approximationsfehler}
		\V_{\mathcal{W}_1 }(\P, \mathbb{Q}) \leq   \V_{\mathcal{W}_2}(\P, \mathbb{Q}) + 2 \inf_{W \in \mathcal{W}_2} \sup_{W^* \in \mathcal{W}_1} \| W-W^*\|_{\infty}.
	\end{align}
	Indeed, the difference $\V_{\mathcal{W}_1 }(\P, \mathbb{Q})-\V_{\mathcal{W}_2}(\P, \mathbb{Q})$ is bounded by
	\begin{align*}
		&\inf_{W \in \mathcal{W}_2} \sup_{W^* \in \mathcal{W}_1}  \Big\{\E\Big[-\log\Big(\frac{1+e^{-W^*(X)}}{2}\Big)-\log\Big(\frac{1+e^{W^*(Y)}}{2}\Big)\Big] \\
        &\qquad\qquad\qquad\quad- \E\Big[-\log\Big(\frac{1+e^{-W(X)}}{2}\Big)-\log\Big(\frac{1+e^{W(Y)}}{2}\Big)\Big]\Big\}\\
		\leq & 	\inf_{W \in \mathcal{W}_2} \sup_{W^* \in \mathcal{W}_1} \Big\{\E\Big[\Big| -\log\Big(\frac{1+e^{-W^*(X)}}{2}\Big) + \log\Big(\frac{1+e^{-W(X)}}{2}\Big)\Big|\Big] \\
		& \qquad\qquad\qquad\quad  +  \E\Big[\Big| -\log\Big(\frac{1+e^{W^*(Y)}}{2}\Big) + \log\Big(\frac{1+e^{W(Y)}}{2}\Big)\Big|\Big]\Big\}\\
		&\leq   \inf_{W \in \mathcal{W}_2} \sup_{W^* \in \mathcal{W}_1}  \big\{   \E[|W^*(X)-W(X)|] +   \E[|W^*(Y)-W(Y)|]\big\}\\
		& \leq 2  \inf_{W \in \mathcal{W}_2} \sup_{W^* \in \mathcal{W}_1}   \|W^* -W \|_{\infty},
	\end{align*} 
	due to Lipschitz continuity of $x\mapsto -\log((1+e^x)/2)$.
	
	From \eqref{Approximationsfehler} we deduce for $\mathcal W\subset\Lip(L,B)$
    \begin{align*} 
		\V_{\Lip(L,B) }(\P, \mathbb{Q}) \leq  \V_{\mathcal{W}}(\P, \mathbb{Q}) + 2 \inf_{W^{\prime} \in \mathcal{W}}\sup_{W \in \Lip(L,B)} \|W-W^{\prime } \|_{\infty}.
	\end{align*}
	
	We abbreviate $\Delta_{\mathcal{W}} \coloneqq \inf_{W^{\prime} \in \mathcal{W}}\sup_{W \in \Lip(L,B)} \|W-W^{\prime } \|_{\infty}.$ Now we can proceed as in \Cref{dreiecksugls}. In particular, it is sufficient to bound $\mathsf{W}_1(\P_n, \P^{\hat G_n(Z)}).$ Due to \Cref{equiv_fast} there is some constant $c>0$ such that for every $G \in \mathcal{G}$
	\begin{align*}
		\mathsf{W}_1(\P_n, \P^{\hat{G}_n(Z)})  & \leq c [ \V_{ \Lip(L,B)}(\P_n, \P^{\hat{G}_n(Z)})]^{1;1/2}\\
		& \leq c [ \V_{ \mathcal{W}}(\P_n, \P^{\hat{G}_n(Z)})+ 2 \Delta_{\mathcal{W}}]^{1;1/2}\\
		& \leq  c [ \V_{\mathcal{W}}(\P_n, \P^{\hat{G}_n(Z)})]^{1;1/2} +  2c [  \Delta_{\mathcal{W}}]^{1;1/2}\\
		& \leq  c [ \V_{\mathcal{W}}(\P_n, \P^{G(Z)})]^{1;1/2} +  2c [  \Delta_{\mathcal{W}}]^{1;1/2}
	\end{align*}
	Because $ \V_{\mathcal{W}}(\P_n, \P^{G(Z)}) \leq \V_{\Lip(L,B) }(\P_n, \P^{G(Z)})$ due to $\mathcal W\subset\Lip(L,B)$, the rest of the proof is identical to the proof of \Cref{dreiecksugls}.
\end{proof}

\begin{proof}[Proof of \Cref{hoelder_oracle}]
	Since $\mathsf{W}_1(\P_n, \P^*)$ can be estimated as in \Cref{dim_red}, we only need to bound $\mathsf{W}_1(\P_n, \P^{\hat{G}_n(Z)})$. 
 For $\Gamma > \max(L, 2B)$, we have $\Lip(L,B) \subset \mathcal{H}^{\alpha}(\Gamma)$, $\alpha\in(0,1)$, and the assumptions of \Cref{equiv_fast} are satisfied. Therefore for every $\alpha \in (0,1)$
 \begin{align*}
     \mathsf{W}_1(\P_n, \P^{\hat{G}_n(Z)}) &\leq c \big[ \V_{\Lip(L,B)} (\P_n, \P^{\hat{G}_n(Z)})\big]^{1;1/2}
      \leq c \big[ \V_{\mathcal{H}^{\alpha}(\Gamma)} (\P_n, \P^{\hat{G}_n(Z)})\big]^{1;1/2}.
 \end{align*}
 Now, \eqref{Approximationsfehler} yields
 \begin{align*}
      \V_{\mathcal{H}^{\alpha}(\Gamma)} (\P_n, \P^{\hat{G}_n(Z)}) \leq  \V_{\mathcal{W}} (\P_n, \P^{\hat{G}_n(Z)}) + 2\Delta_{\mathcal W}\qquad\text{for}\qquad \Delta_{\mathcal W}:=\inf_{W \in  \mathcal{W} }\sup_{W^* \in \mathcal{H}^{\alpha}(\Gamma)} \|W^* - W\|_{\infty}.
 \end{align*}
 Using that $\hat{G}_n$ is the empirical risk minimizer and $\mathcal W\subseteq\mathcal H^\alpha(\Gamma)$, we thus have
 \begin{align*}
     \mathsf{W}_1(\P_n, \P^{\hat{G}_n(Z)}) &\leq c \big[ \V_{ \mathcal{W}} (\P_n, \P^{\hat{G}_n(Z)})\big]^{1;1/2} +c [ \Delta_{\mathcal{W}}]^{1;1/2}\\
     & \leq c \big[ \V_{\mathcal{W}} (\P_n, \P^{G(Z)})\big]^{1;1/2} +c [ \Delta_{\mathcal{W}}]^{1;1/2}\\
     & \leq c \big[ \V_{\mathcal{H}^{\alpha}(\Gamma)} (\P_n, \P^{G(Z)})\big]^{1;1/2} +c [ \Delta_{\mathcal{W}}]^{1;1/2}
 \end{align*}
To bound the first term, we apply \Cref{obere_schranke_allgemein} and $\{-\log(1+e^{-W(\cdot)})\mid W\in\mathcal H^\alpha(\Gamma)\}\subset \mathcal H^\alpha(\Gamma)$ to obtain
\begin{align}
    \V_{\mathcal{H}^{\alpha}(\Gamma)} (\P_n, \P^{G(Z)}) &\leq \sup_{W \in \mathcal{H}^{\alpha}(\Gamma)} \E_{\substack{\hat X \sim \P_n }}[-\log\Big(1+e^{-W(\hat X)} \Big)+\log\Big(1+e^{-W(G(Z))} \Big)] \notag  \\
    & \leq \sup_{W \in \mathcal{H}^{\alpha}(\Gamma)} \E_{\substack{\hat X \sim \P_n \\ }} [W(\hat X)- W(G(Z))] \label{W_Halpha(P_n, P^G)}\\
    & \leq \sup_{W \in \mathcal{H}^{\alpha}(\Gamma)} \E_{\substack{\hat X \sim \P_n } } [W(\hat X)- W(X)]+ \sup_{W \in \mathcal{H}^{\alpha}(\Gamma)} \E [W(X)- W(G(Z))]. \notag
\end{align} 
For the second term we have by Hölder continuity, Jensens inequality and the duality formula of $\mathsf W_1$ as used in the proof of \Cref{dim_red} that
\begin{align*}
    \sup_{W \in \mathcal{H}^{\alpha}(\Gamma)} \E [W(X)- W(G(Z))] 
    &\leq \sup_{W \in \mathcal{H}^{\alpha}(\Gamma)} \E [|W(X)- W(G^*(Z))|]+\sup_{W \in \mathcal{H}^{\alpha}(\Gamma)} \E [|W(G^*(Z))- W(G(Z))|]\\
    & \leq \Gamma\E[|X-G^*(Z)|_p^\alpha]+\Gamma \|G^* - G\|_{\infty}^{\alpha}\\
    & \le \Gamma\mathsf W_1(\P^*,\P^{G^*(Z)})^\alpha+\Gamma \|G^* - G\|_{\infty}^{\alpha}.
\end{align*}
Hence, we have for any $G\in\mathcal G$ and any measurable $G^*\colon\mathcal Z\to\mathcal X$ for some constant $c>0$
\begin{align*}
    \mathsf{W}_1(\P_n, \P^{\hat{G}_n(Z)}) \leq &  c \big[ \sup_{W \in \mathcal{H}^{\alpha}(\Gamma)} \E_{\substack{\hat X \sim \P_n } } [W(\hat X)- W(X)]\big]^{1;1/2} \\
    &\qquad+  c\big[\mathsf W_1(\P^*,\P^{G^*(Z)})^\alpha+ \|G^* - G\|_{\infty}^{\alpha}\big]^{1;1/2}+c [ \Delta_{\mathcal{W}}]^{1;1/2}.
\end{align*}
For the remaining stochastic error term, we first note that
\begin{align*}
    \sup_{W \in \mathcal{H}^{\alpha}(\Gamma)} \E_{\substack{\hat X \sim \P_n } } [W(\hat X)- W(X)] 
    &\leq \sup_{W \in \mathcal{H}^{\alpha}(\Gamma)} \E_{\substack{X_n \sim \P_n } } [W(X_n)- W(G^*(Z))]\\
    &\qquad+ \sup_{W \in \mathcal{H}^{\alpha}(\Gamma)} \E [W(G^*(Z))- W(X)]\\
    &\leq \sup_{W \in \mathcal{H}^{\alpha}(\Gamma)} \E_{\substack{X_n \sim \P_n } } [W(X_n)- W(G^*(Z))]+\Gamma\mathsf W_1(\P^*,\P^{G^*(Z)})^\alpha
\end{align*}
and as in \eqref{eq:proofDimRed} together with \citet[Theorem 4]{Schreuder2020_2} we obtain
\begin{align*}
 \E\Big[\sup_{W \in \mathcal{H}^{\alpha}(\Gamma)} \E_{\substack{X_n \sim \P_n } } [W(X_n)- W(G^*(Z))]\Big] 
 & \leq \E\Big[ \sup_{W \in \mathcal{H}^{\alpha}(\Gamma)} |X- G^*(Z)|_p^\alpha\Big] \\ 
 & \qquad+ \E\Big[\sup_{f \in \mathcal{H}^{\alpha}(M \cdot \Gamma)} \frac{1}{n}\sum_{i = 1}^n  f(Z_i) - \E[f(Z)]\Big]\\
 &\le c\mathsf{W}_1(\mathbb{P^*}, \P^{G^*(Z)})^{\alpha}+c  \begin{cases}n^{-\alpha / d^*}, &  2\alpha<d^* , \\ n^{-1 / 2} \ln (n), & 2 \alpha=d^* , \\ n^{-1 / 2}, &  2\alpha>d^* .\end{cases}
\end{align*}

For the expectation of the first term we use Jensen's inequality 
\begin{equation*}
    \E[|X_i -G^*(Z_i)|_p^{\alpha}] \leq \E[|X_i -G^*(Z_i)|_p]^{\alpha} = \mathsf{W}_1(\mathbb{P^*}, \P^{G^*(Z)})^{\alpha}.\qedhere
\end{equation*}
\end{proof}

\subsection{Proof of \Cref{approx_res_hoelder}}

To prove \Cref{approx_res_hoelder} some additional notation is required. 
The set of locally integrable functions is given by
$$
L_{\text{loc}}^1 (\Omega) \coloneqq \{ f\colon \Omega \rightarrow \mathbb{R} \Big| \int_{K} |f(x)|\;\mathrm{d}x < \infty, \; \text{for all compact } K \subset \Omega^{\circ} \}.
$$
A function $f \in L_{\text{loc}}^1(\Omega)$ has a weak derivative, $D_w^\alpha f$, provided there exists a function $g \in L_{\text{loc}}^1(\Omega)$ such that
$$
\int_{\Omega} g(x) \phi(x) d x=(-1)^{|\alpha|} \int_{\Omega} f(x) \phi^{(\alpha)}(x) d x \quad \text{for all } \phi \in C^{\infty}(\Omega) \text{ with compact support}.
$$
If such a $g$ exists, we define $D_w^\alpha f\coloneqq g.$ For $f \in L_{\text{loc}}^1(\Omega)$ and $k \in \mathbb{N}_0$ the Sobolev norm is
\begin{equation*}
\|f\|_{W^{k,\infty}(\Omega)}:=\max _{|\alpha| \leq k}\left\|D_w^\alpha f\right\|_{\infty, \Omega} .
\end{equation*}
The Sobolev space $W^{k,\infty}(\Omega):= \{ f \in L_{\text{loc}}^1(\Omega):\|f\|_{W_p^k(\Omega)}<\infty \}$ is a Banach space \cite[Theorem 1.3.2]{Brenner_Scott2008}.
For  $f \in W^{k,\infty}(\Omega)$, define the Sobolev semi norm by
$$
|f|_{W^{k,\infty}(\Omega)}\coloneqq \max _{|\alpha|=k}\left\|D_w^\alpha f\right\|_{{\infty}, \Omega} .
$$
Note that $\Lip(L, B,\Omega)\subset W^{1,\infty}(\Omega)$, since $\|f\|_{W^{1,\infty}} \leq \max(L,B)$ for any $f \in \Lip(L, B,\Omega)$. For two normed spaces $(A, \| \cdot\|_A), (B, \| \cdot\|_B)$ we denote the operator norm of a linear operator $T \colon A \rightarrow B$ by
$$
\|T\|:=\sup \big\{\|T x\|_B \;|\; x \in A,\|x\|_A \leq 1\big\}.
$$

\Cref{approx_res_hoelder} is very close to by \cite[Theorem 4.1]{Guehring2020}, which however applies only to functions $f$ which are at least twice (weakly) differentiable. Our proof can thus build on numerous auxiliary results and arguments from \cite{Guehring2020}. We basically keep the proof structure of  \cite{Guehring2020} which in turn relies on \cite{Yarotsky2017}.

Let $d, N \in \mathbb{N}.$ For $m \in \{0,...,N\}^d,$ define the functions $\phi_m \colon\mathbb{R}^d \rightarrow \mathbb{R},$ \begin{equation*}
    \phi_m(x) = \prod_{\ell = 1}^{d} \psi\Big(3N\Big(x_{\ell}- \frac{m_{\ell}}{N}\Big) \Big), \quad \text{where} \quad \psi(x) = \begin{cases}
        1, & |x|<1,\\
        0, & |x|>2, \\
        2-|x|, & 1 \leq |x| \leq 2
    \end{cases}
\end{equation*}
By definition, we have $\|\phi_{m}\|_{\infty}=1  $ for all $m$ and \begin{equation}\label{support_phi}\operatorname{supp} \phi_{m} \subset\left\{x:\left|x_k-\frac{m_k}{N}\right|<\frac{1}{N} \forall k\right\} \eqqcolon B_{\frac{1}{N}, |\cdot|_{\infty}}(\frac{m}{N}).\end{equation}\\
\citet[Lemma C.3 (iv)]{Guehring2020} have verified that $\|\phi_m\|_{W^{1, \infty}(\mathbb{R}^d)} \leq cN$ for some constant $c>0$.

A direct consequence of Lemma 2.11, Lemma C.3, Lemma C.5 and Lemma C.6 by \cite{Guehring2020} is the following approximation result for the localizing functions $\phi_m$ via ReLU networks:
\begin{lemma}\label{approx_phi_with_network}
    For any $\varepsilon \in (0,1/2)$ and any $m \in \{0,...,N\}^d$ there is a network $\Psi_{\varepsilon}$ with ReLU activation function, not more than $C_1\log_2(\varepsilon^{-1})$ layers and no more than $C_2 (N+1)^d\log_2^2(\varepsilon^{-1})$ nonzero weights and no more than neurons $C_3 (N+1)^d(\log_2^2(\varepsilon^{-1}) \vee \log_2(\varepsilon^{-1}))$ such that for $k \in \{0,1\}$ \begin{equation*}
        \| \Psi_{\varepsilon} - \phi_m \|_{W^{k, \infty}} \leq c N^k \varepsilon,
    \end{equation*} where $C_1,C_2,C_3$ and $c$ are constants independent of $m$ and $\varepsilon$. Additionally,
    \begin{equation*}
        \phi_m(x) = 0 \Longrightarrow \Psi_{\varepsilon} (x) = 0,
    \end{equation*}
    and therefore $\operatorname{supp}\Psi_{\varepsilon} \subset B_{\frac{1}{N}, |\cdot|_{\infty}}(\frac{m}{N}).$
\end{lemma}
Next we approximate a bounded Lipschitz function using linear combinations of the set $\{\phi_m\colon m \in \{1,...,N\}^d\}$. The approximation error will be measured in the Hölder norm from \eqref{eq:Hoeldernorm}. 
\begin{lemma}  \label{approx_function_using_phi}
    Let $0 <\alpha < 1.$ There exists a constant $C_1>0$ such that for any $f \in W^{1, \infty}((0,1)^d)$ there are constants $c_{f,m}$ for $m \in \{0,...,N\}^d$ such that
    \begin{equation*}
        \Big\| f- \sum_{m \in \{0,...,N\}^d}c_{f,m} \phi_m \Big\|_{\mathcal{H}^{\alpha}} \leq C_1 \Big(\frac{1}{N}\Big)^{1-\alpha}\| f\|_{W^{1, \infty }}.
    \end{equation*}
    The coefficients satisfy for a $C_2 >0$
    \begin{equation*}
        |c_{f,m}| \leq C_2 \|\tilde{f}\|_{W^{1, \infty}(\Omega_{m, N})},
    \end{equation*} where $\Omega_{m, N} \coloneqq B_{\frac{1}{N}, |\cdot|_{\infty}}(\frac{m}{N})$ and $\tilde{f}\in W^{1, \infty}(\mathbb{R}) $ is an extension of $f.$
\end{lemma}
\begin{proof}
Let $E\colon W^{1, \infty} ((0,1)^d) \rightarrow W^{1, \infty}(\mathbb{R})$ be the continuous linear extension operator from \cite[Theorem 5]{stein1970} and set $\tilde{f} \coloneqq Ef.$ As $E$ is continuous there exists a $C_E>0$ such that
\begin{equation*}
    \|\tilde{f} \|_{W^{1, \infty}(\mathbb{R}^d)} \leq C_E \|f \|_{W^{1, \infty}}.
\end{equation*}
    \textbf{Step 1 }(Choice of $c_{f,m}$): For each $m \in\{0, \ldots, N\}^d$ we define
	$$
	c_{f, m} = \int_{B_{m, N}} \tilde{f}(y) \rho(y) \ \mathrm{d}y \text { for } B_{m, N}:=B_{\frac{3}{4 N}, \mid \cdot \mid }\Big(\frac{m}{N}\Big)
	$$
	and an arbitrary cut-off function $\rho$ supported in $B_{m, N}$, i.e.
	\begin{center}$\rho \in C_c^{\infty}(\mathbb{R}^d) \quad$ with $\quad \rho(x) \geq 0$ for all $x \in \mathbb{R}^d, \quad \operatorname{supp} \rho=B_{m, N} \quad$ and $\quad \int_{\mathbb{R}^n} \rho(x) d x=1$. \end{center} 
	
	Then \[
	|c_{m,f}| = \Big|\int_{B_{m, N}}\tilde{f}(y) \rho(y) \ \mathrm{d}y \Big|\leq \|\tilde{f}\|_{\infty,\Omega_{m, N}} \int_{B_{m, N}}  \rho(y) \ \mathrm{d}y = \|\tilde{f}\|_{{\infty},\Omega_{m, N}} \leq C_E\|f\|_{W^{1, \infty}(\Omega_{m, N})} .
	\]
 \textbf{Step 2} (Local estimates in $\|\cdot\|_{W^{k, p}}$): The coefficients $c_{m,f}$ are the averaged Taylor polynomials in the sense of \citet[Definition 4.1.3]{Brenner_Scott2008} of order $1$ averaged over $B_{m,N}.$ As \citet[Proof of Lemma C.4, Step 2]{Guehring2020} showed, the conditions of the Bramble-Hilbert-Lemma \cite[Theorem 4.3.8]{Brenner_Scott2008} are satisfied. Hence for $k \in \{0,1\}$
 \begin{equation*}
     |\tilde{f}-c_{m,f}|_{W^{k, \infty}(\Omega_{m,N})} \leq C_1 \Big(\frac{2 \sqrt{d}}{N}\Big)^{1-k} |\tilde{f}|_{W^{1, \infty}(\Omega_{m;n})} \leq C_2 \Big(\frac{1}{N}\Big)^{1-k} \|\tilde{f}\|_{W^{1, \infty}(\Omega_{m;n})}.
 \end{equation*}
 Now using $\phi_m$ as defined above, we get \begin{align}
		\left\|\phi_m\Big(\tilde{f}-c_{f, m}\Big)\right\|_{\infty,\Omega_{m, N}} &\leq\left\|\phi_m\right\|_{\infty,\Omega_{m, N}} \cdot\left\|\tilde{f}-c_{f, m}\right\|_{\infty,\Omega_{m, N}}
		 \leq C_2\frac{1}{N}\|\tilde{f}\|_{W^{1, \infty}(\Omega_{m, N})} 	\label{Local_L_infty}
	\end{align}
	Due to the product inequality for weak derivatives \cite[Lemma B.6]{Guehring2020} there is a constant $C^{\prime}>0$ such that the supremum norm of the weak derivative  is bounded by
	\begin{align}
		\left|\phi_m\Big(\tilde{f}-c_{f,m}\Big)\right|_{W^{1, \infty}(\Omega_{m, N})} \leq & C^{\prime}|\phi_m|_{W^{1, \infty}(\Omega_{m, N})} \cdot\left\|\tilde{f}-c_{f, m}\right\|_{\infty,\Omega_{m, N}}\notag \\
		& +C^{\prime}\left\|\phi_m\right\|_{\infty,\Omega_{m, N}} \cdot|\tilde{f}-c_{f, m}|_{W^{1, \infty}(\Omega_{m, N})}\notag \\
		\leq & C^{\prime} \cdot c N \cdot C_2\frac{1}{N}\|\tilde{f}\|_{W^{1, \infty}(\Omega_{m, N})}+C^{\prime} \cdot C_3\|\tilde{f}\|_{W^{1, \infty}(\Omega_{m, N})} \notag \\
		= & C_4\|\tilde{f}\|_{W^{1, \infty}(\Omega_{m, N})}. \label{Local_weak}
	\end{align}
	Combining \eqref{Local_L_infty} and \eqref{Local_weak}  we get
	\[
	\left\|\phi_m\Big(\tilde{f}-c_{f, m}\Big)\right\|_{W^{1, \infty}(\Omega_{m, N})} \leq C_5 \|\tilde{f}\|_{W^{1, \infty}(\Omega_{m, N})}.
	\]
 \textbf{Step 3} (Global estimate in $\|\cdot\|_{W^{k, p}}$): As $\sum_{m \in \{0,...,N\}^d} \phi_m =1,$ we have that
 $$
	\tilde{f}(x)=\sum_{m \in\{0, \ldots, N\}^d} \phi_m(x) \tilde{f}(x), \quad \text { for a.e. } x \in(0,1)^d .
	$$
 As $\tilde{f}\big|_{(0,1)^d} = f$ we have for $k \in \{0,1\}$
  \begin{align}
		\Big\|f- \sum_{m \in\{0, \ldots, N\}^d} \phi_m c_{f, m}\Big\|_{W^{k, \infty}((0,1)^d)} & = 
  \Big\|\tilde{f}- \sum_{m \in\{0, \ldots, N\}^d} \phi_m c_{f, m}\Big\|_{W^{k, \infty}((0,1)^d)}\notag
  \\& =\Big\|\sum_{m \in\{0, \ldots, N\}^d} \phi_m\big(\tilde{f}-c_{f, m}\big)\Big\|_{W^{k, \infty}((0,1)^d)}\notag \\
		& \leq \sup_{\widetilde{m} \in\{0, \ldots, N\}^d} \ \Big\| \sum_{m \in\{0, \ldots, N\}^d} \phi_m\big(\tilde{f}-c_{f, m}\big) \Big\|_{W^{k, \infty}(\Omega_{\widetilde{m}, N})} \label{glob_begin}
	\end{align}
 where the last step follows from $(0,1)^d \subset \bigcup_{\widetilde{m} \in\{0, \ldots, N\}^d} \Omega_{\widetilde{m}, N}$.
Now we obtain for each $\widetilde{m} \in\{0, \ldots, N\}^d$ using \eqref{support_phi}, \eqref{Local_L_infty} and \eqref{Local_weak}
	\begin{align*}
		\Big\|\sum_{m \in\{0, \ldots, N\}^d} \phi_m\big(\tilde{f}-c_{f, m}\big)\Big\|_{W^{k, \infty}(\Omega_{\tilde{m}, N})} &  \leq \sup_{\substack{m \in\{0, \ldots, N\}^d \\|m-\widetilde{m}|_{\infty} \leq 1}}\big\|\phi_m\big(\tilde{f}-c_{f, m}\big)\big\|_{W^{k, \infty}(\Omega_{\tilde{m}, N})} \\
		&	\leq \sup_{\substack{m \in\{0, \ldots, N\}^d \\|m-\widetilde{m}|_{\infty} \leq 1}}\big\|\phi_m\big(\tilde{f}-c_{f, m}\big)\big\|_{W^{k, \infty}(\Omega_{m, N})} \\
		& \leq C_6\Big(\frac{1}{N}\Big)^{1-k} \sup_{\substack{m \in\{0, \ldots, N\}^d,\\|m-\widetilde{m}|_{\infty} \leq 1}}\|\tilde{f}\|_{W^{1, \infty}(\Omega_{m, N})}.
	\end{align*}
 Plugging this into \eqref{glob_begin}, we obtain for $k \in \{0,1\}$ \begin{align}
		\Big\|f-\sum_{m \in\{0, \ldots, N\}^d} \phi_m c_{f, m}\Big\|_{W^{k, \infty}((0,1)^d)}& \leq C_6   \Big(\frac{1}{N}\Big)^{(1-k)} \sup_{\tilde{m} \in\{0, \ldots, N\}^d}  \Big(\sup_{\substack{m \in\{0, \ldots, N\}^d,\\|m-\widetilde{m}|_{\infty} \leq 1}}\|\tilde{f}\|_{W^{1, \infty}(\Omega_{m, N})}\Big) \notag\\
		& \leq C_7  \Big(\frac{1}{N}\Big)^{(1-k)}  \sup_{\widetilde{m} \in\{0, \ldots, N\}^d}\|\tilde{f}\|_{W^{1, \infty}(\Omega_{\widetilde{m}, N})}\notag\\
		& \leq C_8 	 \Big(\frac{1}{N}\Big)^{(1-k)} \| \tilde{f}\|_{W^{1, \infty}(\mathbb{R}^d) }\notag \\
		& \leq C_{9}  \Big(\frac{1}{N}\Big)^{(1-k)} \| f\|_{W^{1, \infty}((0,1)^d) } .\label{global_k_bound}
  \end{align}
  \textbf{Step 4} (Interpolation): 
  Define the linear operators $T_0\colon W^{1, \infty}((0,1)^d) \rightarrow L^{\infty}((0,1)^d), T_{\alpha}\colon W^{1, \infty}((0,1)^d) \rightarrow\mathcal{H}^{\alpha}((0,1)^d)$ and $ T_1\colon W^{1, \infty} ((0,1)^d) \rightarrow W^{1, \infty} ((0,1)^d) $ via
  \begin{align*}
 T_k(f) = f-\sum_{m \in\{0, \ldots, N\}^d} \phi_m c_{f, m}, \quad k \in \{0,\alpha ,1\}
  \end{align*}
  Note that the linearity follows from the definition of the constants $c_{f,m}.$ Using \citet[Theorem 1.6]{lunardi}, for the nontrivial interpolation couple see \cite[p.11 f.]{lunardi}, leads to 
  \begin{equation*}
      \| T_{\alpha} \| \leq \|T_0\|^{1-\alpha} \|T_1\|^{\alpha}. 
  \end{equation*}
  Note that $\|\cdot\|_{\mathcal{H}^{\alpha}}$ is equivalent to $\|\cdot\|_{W^{s, \infty}(\Omega)}$ in \citet{Guehring2020}.
 Using \eqref{global_k_bound} we conclude
 \begin{align*}
     \Big\|f-\sum_{m \in\{0, \ldots, N\}^d} \phi_m c_{f, m}\Big\|_{\mathcal{H}^{\alpha}} &\leq C_{10} \Big(\frac{1}{N}\Big)^{1-\alpha} \| f\|_{W^{1, \infty} }.\qedhere
 \end{align*}
\end{proof}

Now we want to approximate the function $\sum_{m \in \{0,...,N\}^d}c_{f,m} \phi_m $ in Hölder norm using a ReLU network. 
\begin{lemma} \label{approx_sum_phi_with_network}
    For any $\varepsilon \in (0, 1/2)$ there is a neural network $\Phi_{\varepsilon}$ with ReLU activation function such that for  $(c_{f, m})_m$ from \Cref{approx_function_using_phi}, there is a constant $C>0$ such that
    \begin{equation*}
	\Big\|\sum_{m \in\{0, \ldots, N\}^d} \phi_m c_{f,m}-\Phi_{\varepsilon}\Big\|_{\mathcal{H}^{\alpha}} \leq C\|f\|_{W^{1, \infty}} N^\alpha \varepsilon,
    \end{equation*}
    the number of layers is at most $\lceil C_1\log_2(\varepsilon^{-1})\rceil ,$ the number of nonzero weights is at most $\lceil C_2 (N+1)^d\log_2^2(\varepsilon^{-1})\rceil$ and the number of neurons is at most $\lceil C_3 (N+1)^d(\log_2^2(\varepsilon^{-1}) \vee \log_2(\varepsilon^{-1}))\rceil,$ with $C_1, C_2$ and $C_3$ from \Cref{approx_phi_with_network}.
\end{lemma}
\begin{proof}
    From \Cref{approx_phi_with_network} we know that there are neural networks $\Psi_{\varepsilon,m}$ with at most $\lceil C_1\log_2(\varepsilon^{-1})\rceil $ layers,  $\lceil C_2 (N+1)^d\log_2^2(\varepsilon^{-1})\rceil$ nonzero weights and $\lceil C_3 (N+1)^d(\log_2^2(\varepsilon^{-1}) \vee \log_2(\varepsilon^{-1}))\rceil$ neurons  that approximate $\phi_m$ such that for $k \in \{0,1\}$ \[
	\left\|\phi_m-\Psi_{\varepsilon,m}\right\|_{W^{k, \infty}} \leq c^{\prime} N^k \varepsilon.
	\] 
    Now we parallelize these networks and multiply with the coefficients $c_{f,m}$ afterwards. Hereby, we construct a network $\Phi_{\varepsilon}$ with $1+\lceil C_1\log_2(\varepsilon^{-1})\rceil $ layers,  $N^d +\lceil C_2 (N+1)^d\log_2^2(\varepsilon^{-1})\rceil$ nonzero weights and $1+ \lceil C_3 (N+1)^d(\log_2^2(\varepsilon^{-1}) \vee \log_2(\varepsilon^{-1}))\rceil$ neurons  such that \begin{equation} \label{approx_phi_NN}
	    \Phi_{\varepsilon} = \sum_{m \in \{0,...,N\}^d} c_{f, m }\Psi_{\varepsilon,m}. 
	\end{equation}
    For each $m \in \{0,...,N\}^d$ denote $\Omega_{m, N}=B_{\frac{1}{N},|\cdot|_{\infty}}(\frac{m}{N})$ as above. For $k \in \{0,1\}$ we get
    $$
	\begin{aligned}
		 &\Big\|\Phi_{\varepsilon} - \sum_{m \in\{0, \ldots, N\}^d} c_{f,m} \phi_m \Big\|_{W^{k, \infty}((0,1)^d)}  =\Big\|  \sum_{m \in\{0, \ldots, N\}^d} c_{f,m} (\Psi_{\varepsilon,m} - \phi_m )\Big\|_{W^{k, \infty}((0,1)^d)} \\
		&\qquad\qquad  \leq \sup_{\widetilde{m} \in\{0, \ldots, N\}^d}  \Big\|  \sum_{m \in\{0, \ldots, N\}^d} c_{f,m} (\Psi_{\varepsilon,m} - \phi_m )\Big\|_{W^{k, \infty}(\Omega_{\widetilde{m}, N} \cap(0,1)^d )}\\
  &\qquad\qquad  \leq 3^d   \sup_{\widetilde{m} \in\{0, \ldots, N\}^d}  \sup_{m \in\{0, \ldots, N\}^d}\Big\| c_{f,m} (\Psi_{\varepsilon,m} - \phi_m )\Big\|_{W^{k, \infty}(\Omega_{\widetilde{m}, N} \cap(0,1)^d )}  \\
	&\qquad\qquad  \leq  3^d   \sup_{\widetilde{m} \in\{0, \ldots, N\}^d}  \sup_{m \in\{0, \ldots, N\}^d}|c_{f,m}| \Big\|  (\Psi_{\varepsilon,m} - \phi_m )\Big\|_{W^{k, \infty}(\Omega_{\widetilde{m}, N} \cap(0,1)^d )}  \\ 
 &\qquad\qquad  \leq 3^d   \sup_{\widetilde{m} \in\{0, \ldots, N\}^d}  \sup_{m \in\{0, \ldots, N\}^d} \|\tilde{f}\|_{W^{1, \infty}(\Omega_{m, N})}\| \Psi_{\varepsilon,m} - \phi_m\|_{W^{k, \infty}(\Omega_{\widetilde{m}, N} \cap(0,1)^d )}\\
		& \qquad\qquad \leq C  N^k \varepsilon \| f\|_{W^{1, \infty}}.
	\end{aligned} $$
	  The second to last inequality follows from the fact that on $\Omega_{\widetilde{m}, N}$ is within the support of $\phi_m$ only for $|m-\widetilde{m}|_{\infty} \leq 1.$ The last inequality follows from \eqref{approx_phi_NN} and the continuity of the extension operator, see \cite[Theorem 5]{stein1970}.
   As in Step 4 of \Cref{approx_function_using_phi}, we conclude using \citet[Theorem 1.6]{lunardi}
   \begin{equation*}
	\Big\|\Phi_{\varepsilon}-\sum_{m \in\{0, \ldots, N\}^d} \phi_m c_{f,m}\Big\|_{\mathcal{H}^{\alpha}} \leq C N^\alpha \varepsilon\|f\|_{W^{1, \infty}}.\qedhere
   \end{equation*}
\end{proof}

Now we are ready to proof \Cref{approx_res_hoelder}.
\begin{proof}[Proof of \Cref{approx_res_hoelder}]
    Combining \Cref{approx_function_using_phi} and \Cref{approx_sum_phi_with_network} with $\|f\|_{W^{1, \infty}} \leq B $ yields for a constant $C>0$ for any $\tilde\varepsilon\in(0,1/2)$ that
    \begin{align}
        \|f- \Phi_{\tilde\varepsilon}\|_{\mathcal{H}^{\alpha}} &\leq  \Big\|f- \sum_{m \in\{0, \ldots, N\}^d} c_{f,m} \phi_m\Big\|_{\mathcal{H}^{\alpha}}  +  \Big\|\sum_{m \in\{0, \ldots, N\}^d} c_{f,m} \phi_m- \Phi_{\tilde\varepsilon}\Big\|_{\mathcal{H}^{\alpha}} \notag \\
        & \leq CB\Big( \Big(  \frac{1}{N}\Big)^{1-\alpha} + N^{\alpha} \tilde\varepsilon\Big), \label{dreieckugl_netzwerkapprox}
    \end{align} where $\tilde\varepsilon$ determines the approximation accuracy  in \Cref{approx_sum_phi_with_network}.
    For
    $$
	N:=\left\lceil\Big(\frac{\varepsilon}{2 C B}\Big)^{-1 /(1-\alpha)}\right\rceil,
	$$
 we get for the first term in \eqref{dreieckugl_netzwerkapprox} \begin{equation*}
     \Big(\frac{1}{N}\Big)^{1-\alpha} \leq \frac{\varepsilon}{2CB}.
 \end{equation*}
 Choosing \begin{equation}\label{eps_NN}
     \tilde\varepsilon = \frac{\varepsilon}{2CB} \Big(\Big( \frac{\varepsilon}{2CB}\Big)^{-\frac{1}{1-\alpha}} +1\Big)^{-\alpha}
 \end{equation}
 leads to 
 \begin{align*}
        \|f- \Phi_{\tilde\varepsilon}\|_{\mathcal{H}^{\alpha}} &\leq \varepsilon.
        \end{align*}
        From \Cref{approx_phi_with_network} we know that there is a ReLU network with no more than $1+\lceil C_1\log_2(\tilde\varepsilon^{-1})\rceil $ layers,  $N^d +\lceil C_2 (N+1)^d\log_2^2(\tilde\varepsilon^{-1})\rceil$ nonzero weights and $1+ \lceil C_3 (N+1)^d(\log_2^2(\tilde\varepsilon^{-1}) \vee \log_2(\tilde\varepsilon^{-1}))\rceil$ neurons with the required properties. Inserting \eqref{eps_NN} and assuming $CB>\frac{1}{2}$ yields
  \begin{align*}
      \log_2(\tilde\varepsilon^{-1} ) &\le \log_2\Big(  \frac{2CB}{\varepsilon} 2^{\alpha}  \Big( \frac{\varepsilon}{2CB}\Big)^{-\frac{\alpha}{1-\alpha}}\Big)
        \leq C^{\prime} \log_2( \varepsilon^{-\frac{1}{1-\alpha}}).
  \end{align*}
Thus there are $C^{\prime},C^{\prime \prime}$ and $C^{\prime\prime\prime}$ such that the ReLU network has no more than $1 + \lceil C^{\prime} \log_2(\varepsilon^{-\frac{1}{1-\alpha}})\rceil $ layers, $\lceil C^{\prime \prime} \varepsilon^{-\frac{d}{1-\alpha}}\log^2_2(\varepsilon^{-\frac{1}{1-\alpha}})\rceil $ nonzero weights and $1+ \lceil C^{\prime \prime \prime}   \varepsilon^{-\frac{d}{1-\alpha}}(\log_2^2(\varepsilon^{-\frac{1}{1-\alpha}}) \vee \log_2(\varepsilon^{-\frac{1}{1-\alpha}}))\rceil$  neurons. \\

Since $f\in\Lip(L,B)\subseteq \mathcal H^{\alpha}(\Gamma)$ for $\Gamma=\max(L,2B)$, we conclude
\begin{equation*}
    \|\Phi_{\tilde\varepsilon}\|_{\mathcal{H}^{\alpha}}\le\|f\|_{\mathcal{H}^{\alpha}}+\|\Phi_{\tilde\varepsilon}-f\|_{\mathcal{H}^{\alpha}}\le \Gamma+\varepsilon.
\end{equation*}
\Cref{alles_eingesetzt} is a straightforward combination of \Cref{hoelder_oracle} and \Cref{approx_res_hoelder}.
\end{proof}

\subsection{Proof of \Cref{thm:wasserstein_gan}}
\begin{proof}[Proof of \Cref{thm:wasserstein_gan}]    
    
    First we note that for $\Gamma > 1,$ there is an $L>0,$ such that there is  a $B>0$ with $2B<\Gamma -1$ and with $\hat{X} \sim \mathbb{P}^*$
    \begin{align*}
        \sup_{W \in \operatorname{Lip}(L)}\mathbb{E}[W(\hat{X}) -W(\hat{G}_n(Z))]= \sup_{W \in \operatorname{Lip}(L, 2B)}\mathbb{E}[W(\hat{X}) -W(\hat{G}_n(Z))].
    \end{align*}
This $L>0$ exists as $[0,1]^d$ is bounded and adding a constant to any function $W \in \operatorname{Lip}(L)$ will not change the value of $\mathbb{E}[W(\hat{X}) -W(\hat{G}_n(Z))].$
    
 Then we get for every $G \in \mathcal{G}$ with the same reasoning as in the proof of \Cref{dreiecksugls_diskriminierer}
    \begin{align*}
\mathsf{W}_1(\mathbb{P}^*, \mathbb{P}^{\hat{G}_n(Z)}) & \leq \mathsf{W}_1(\mathbb{P}^*, \mathbb{P}_n) + \mathsf{W}_1(\mathbb{P}_n, \mathbb{P}^{\hat{G}_n(Z)}) \\
& =  \mathsf{W}_1(\mathbb{P}^*, \mathbb{P}_n) + \frac{1}{L} \mathsf{W}_L(\mathbb{P}_n, \mathbb{P}^{\hat{G}_n(Z)}) \\
& \leq   \mathsf{W}_1(\mathbb{P}^*, \mathbb{P}_n) + \frac{1}{L} \mathsf{W}_{\mathcal{W}}(\mathbb{P}_n, \mathbb{P}^{\hat{G}_n(Z)}) + \frac{2}{L} \inf_{W \in \mathcal{W}} \sup_{W^{\prime}\in \operatorname{Lip}(L,2B)}\|W-W^{\prime}\|_{\infty}  \\
& \leq \mathsf{W}_1(\mathbb{P}^*, \mathbb{P}_n) + \frac{1}{L} \mathsf{W}_{\mathcal{W}}(\mathbb{P}_n, \mathbb{P}^{G(Z)}) + \frac{2}{L} \inf_{W \in \mathcal{W}} \sup_{W^{\prime}\in \operatorname{Lip}(L,2B)}\|W-W^{\prime}\|_{\infty}  \\
& \leq \mathsf{W}_1(\mathbb{P}^*, \mathbb{P}_n) + \frac{1}{L} \mathsf{W}_{\mathcal{H}^{\alpha}}(\mathbb{P}_n, \mathbb{P}^{G(Z)}) + \frac{2}{L} \inf_{W \in \mathcal{W}} \sup_{W^{\prime}\in \operatorname{Lip}(L,2B)}\|W-W^{\prime}\|_{\infty}  \\
    \end{align*}
 
    The bound on $\mathsf{W}_{\mathcal{H}^{\alpha}}(\mathbb{P}_n, \mathbb{P}^{G(Z)}) $ depending on the intrinsic dimension $d^*$ was already derived in \Cref{hoelder_oracle} (starting with Equation \eqref{W_Halpha(P_n, P^G)}). The bound on $\mathsf{W}_1(\mathbb{P}^*, \mathbb{P}_n)$ depending on the intrinsic dimension $d^*$ was already derived in \Cref{rate_schreuder}.

\end{proof}

\subsection{Calculations for \Cref{bsp_lin}}
\label{proof_bsp_lin}
For the Wasserstein distance we get $\mathsf{W}_1(\mathbb{P}, \mathbb{Q}) = \gamma.$
	The Vanilla GAN distance using all Lipschitz $L$ affine functions as discriminator yields in this example $\V_{a\cdot +b}(\P, \mathbb Q)=\max_{\substack{a,b\in\R,\\ |a| \leq L}}f(a,b)$ for
\begin{align*}
	f(a,b) \coloneqq \frac{1}{2} \big(-\log\big(1+e^{-a\gamma -b}\big)-\log\big(1+e^{-a(\gamma + \varepsilon) -b}\big) -\log\big(1+e^{b}\big)-\log\big(1+e^{a\varepsilon +b}\big)\big) + \log(4).
\end{align*}
	Standard calculus yields for fixed $a$ the unique maximizer $b^* =  - \frac{a(\varepsilon + \gamma)}{2}$ and
\begin{align*}
	f(a,b^*) = -\log\Big(1+e^{-\frac{a(\varepsilon+ \gamma)}{2}}\Big)-\log\Big(1+e^{\frac{a(\varepsilon- \gamma)}{2}}\Big) + \log(4).
\end{align*}
Since
\begin{align*}
	\frac{\partial }{\partial a}f(a,b^*) =  \frac{\varepsilon +\gamma }{2(e^{\frac{a(\gamma + \varepsilon)}{2}}+1)} - \frac{\varepsilon -\gamma}{2(e^{-\frac{a(\varepsilon - \gamma)}{2}}+1)},
\end{align*}
for $\varepsilon \leq \gamma,$ the maximizing $a$ is maximal $a^* =L.$ This coincides with the intuitive choice: as the support of $\P^X$ and the support of  $\P^Y$ can be separated by a single point on $\mathbb{R},$ we expect the optimal discriminator to be affine linear. Standard calculus yields the linear upper and lower bound for $\varepsilon = \frac{1}{4}$.\\
For $\varepsilon > \gamma,$ the unrestricted maximizing $a^*$ solves the equation \begin{align*}
	(\varepsilon - \gamma )e^{\frac{a^*(\varepsilon + \gamma)}{2}}- 	(\varepsilon + \gamma )e^{-\frac{a^*(\varepsilon - \gamma)}{2}} = 2\gamma.
\end{align*}
While there is no closed form solution, a numerical approximation (for $\varepsilon = \frac{1}{4}$) yields for $\gamma < \varepsilon$ and $L>16$ such that $a^*$ is feasible
\begin{align*}
	\frac{\mathsf{W}_1(\P^X, \P^Y)^2}{2} \leq \V_{a\cdot +b} (\P^X, \P^Y) \leq a \cdot\mathsf{W}_1(\P^X, \P^Y)^2.
\end{align*}

\bibliography{literatur}

\end{document}